\documentclass[11pt]{article}

\usepackage{amsmath,amssymb,amsthm}
\usepackage{mathtools}
\usepackage{geometry}
\usepackage{hyperref}
\usepackage{enumitem}
\usepackage{xcolor}
\newtheorem{theorem}{Theorem}[section]
\newtheorem{lemma}[theorem]{Lemma}
\newtheorem{proposition}[theorem]{Proposition}
\newtheorem{corollary}[theorem]{Corollary}
\newtheorem{definition}[theorem]{Definition}
\newtheorem{remark}[theorem]{Remark}

\title{Spectral distribution of Jacobi weighted histopolation matrices via GLT theory}

\author{A. Guessab, F. Nudo, S. Serra-Capizzano}

\date{}

\begin{document}
\maketitle

\begin{abstract}

In this paper we study a weighted histopolation problem on $[-1,1]$ associated with Jacobi weights.
In the first part of the present work we prove results in approximation theory,
while in the second we analyze the resulting matrices from an asymptotic linear
algebra perspective. More in detail, in the first part, given weighted cell averages, we construct a reconstruction operator based on weighted primitives of Jacobi polynomials and investigate the resulting discretization matrices. At any fixed discretization level, we derive an exact factorization of the
histopolation matrix through a backward-difference operator and a sampling
operator of Jacobi weighted primitives. Combining a sharp integration by parts
identity with the three-term recurrence of Jacobi polynomials, we further show
that the primitive sampling operator admits an explicit decomposition involving
a tridiagonal coupling matrix in the Jacobi spectral index. This yields a tridiagonal factor representation of the histopolation matrix. In the second part, under standard mesh-regularity assumptions, we show that all the various induced matrix sequences belong to the Generalized Locally Toeplitz (GLT) class, by describing in detail the related GLT symbols. As a consequence, we provide the corresponding spectral distributions and
discuss their implications for numerical stability when solving the associated
linear systems.
\end{abstract}

\section{Introduction}
Approximation from averaged data is a natural alternative to classical
interpolation when pointwise samples are either unavailable or not the most
appropriate form of information. In many applications, the measured quantities
are integrals over cells, faces, intervals, or more general geometric
regions~\cite{Kak:2001:POC,Natterer:2001:TMO,Palamodov:2016:RFI,Bosner:2020:AOC}.
In such situations, the reconstruction procedure should be consistent with the
integral nature of the data. This is the basic idea behind histopolation
methods. In these methods, the degrees of freedom are not point evaluations,
but average values or moment-type functionals. In recent years, several
polynomial histopolation methods have been introduced and studied~\cite{Bruni:2025:PHO,Bruno:2026:BPH}. When the degrees of freedom are defined
by weighted integral functionals, this naturally leads to \emph{weighted
histopolation}~\cite{DellAccio:2026:AGP,Nudo:2026:FRU,DellAccio:2026:NAM}. This setting appears in
several contexts, including constrained approximation, weighted mean-value
problems, geometric design, and reconstruction from averaged data; see, e.g.,
\cite{Demichelis:1995:GAO,Bose:1965:FSA,Diagana:2011:TEO,Pesenson:2019:ASA,Pesenson:2019:WSA}.
From a mathematical point of view, the presence of a weight may affect basic
structural properties of the histopolation problem, including unisolvence and
stability~\cite{Guessab:2025:QWH}. After choosing a basis for the
reconstruction space, the weighted moment conditions lead to structured
matrices, whose asymptotic behaviour as the matrix sizes tend to infinity is not immediate, even when unisolvence is
guaranteed at each fixed discretization level.

In this paper we consider a weighted histopolation problem on the interval
$[-1,1]$ associated with Jacobi weights
\[
\omega_{\alpha,\beta}(t)=(1-t)^\alpha(1+t)^\beta, \quad t\in[-1,1].
\]
This class of weights is particularly natural, since it is associated with a
large family of orthogonal polynomials and provides useful weighted estimates.
More precisely, starting from weighted cell averages on a partition of
$[-1,1]$, we reconstruct a polynomial in a Jacobi-type basis. We use
weighted primitives of Jacobi polynomials to derive an exact factorization of
the histopolation matrix. This factorization links the approximation problem with its asymptotic matrix
analysis. More precisely, we prove that the histopolation matrix can be written
as the product of a backward-difference operator and a primitive sampling
matrix. This representation allows us to separate the asymptotically relevant
component from a term that is negligible in the singular value distribution sense. The asymptotic analysis is performed by means of Generalized Locally Toeplitz (GLT) theory, which provides a general framework for describing the singular
value and eigenvalue distributions of structured matrix sequences arising from
numerical discretizations; see
\cite{Garoni:2018:GLT1,Garoni:2018:GLT2}
for the theory and
\cite{Dorostkar:2016:SAO,Salinelli:2016:EES,Garoni:2019:SBA,Benedusi:2024:MEC}
for related applications. A further goal of the present paper is to provide a GLT analysis of the matrix sequences
generated by weighted Jacobi histopolation. Under standard mesh-regularity
assumptions, we identify the relevant GLT symbols and derive the corresponding
singular value distributions. We also prove stability bounds in a mesh-weighted
discrete norm, valid for arbitrary partitions and for any discretization
level.

The paper is organized as follows. In Section~\ref{sec2} we formulate the
weighted histopolation problem on $[-1,1]$ and recall the required estimates for
Jacobi polynomials. We introduce weighted primitive
functions and derive an exact factorization of the histopolation matrix. In
Section~\ref{sec3} we recall the GLT notation and tools used throughout the
asymptotic analysis. In Section~\ref{sec4} we analyze the matrix sequences
generated by the weighted Jacobi histopolation problem. We identify negligible
terms, compute the relevant GLT symbols, derive the corresponding singular
value distributions, and prove stability bounds in a mesh-weighted discrete
norm. Section~\ref{sec5} contains the numerical experiments, which confirm the
theoretical results, and Section~\ref{sec6} is devoted to conclusions and to open problems.

\section{Weighted Jacobi histopolation}\label{sec2}

In the current section we formulate the weighted histopolation problem associated with
Jacobi weights and present the basic approximation results on Jacobi
polynomials needed in the sequel. We start by recalling a few standard identities and weighted estimates for Jacobi polynomials.

\subsection{Jacobi polynomial estimates}
Let $\alpha,\beta>-1$, and let $P^{(\alpha,\beta)}_j$ denote the Jacobi polynomial
of degree $j$ on $[-1,1]$, orthogonal with respect to the weight
\begin{equation}\label{jacweight}
\omega_{\alpha,\beta}(t)=(1-t)^{\alpha}(1+t)^{\beta}.
\end{equation}
We shall use the following classical identities for Jacobi polynomials. We state
them here to fix the notation and refer the reader
to~\cite{Szego:1975:OP,Milovanovic:1997:OPS} for further details.

 The following differentiation formula holds
\begin{equation}\label{eq:jacobi-derivative}
\frac{d}{dt}P^{(\alpha,\beta)}_j(t)
=
\frac{1}{2}(j+\alpha+\beta+1)
P^{(\alpha+1,\beta+1)}_{j-1}(t),
\quad
P_{-1}^{(\alpha,\beta)}=0,
\end{equation}
together with the three-term recurrence relation
\begin{equation}\label{eq:jacobi-recurrence}
t P^{(\alpha,\beta)}_j(t)
=
a_j P^{(\alpha,\beta)}_{j+1}(t)
+
b_j P^{(\alpha,\beta)}_j(t)
+
c_j P^{(\alpha,\beta)}_{j-1}(t), \quad j\ge1,
\end{equation}
where
\begin{align}
a_j &=
\frac{2(j+1)(j+\alpha+\beta+1)}
{(2j+\alpha+\beta+1)(2j+\alpha+\beta+2)}, \label{aj}\\
b_j &=
\frac{\beta^2-\alpha^2}
{(2j+\alpha+\beta)(2j+\alpha+\beta+2)}, \label{bj}\\
c_j &=
\frac{2(j+\alpha)(j+\beta)}
{(2j+\alpha+\beta)(2j+\alpha+\beta+1)}. \label{cj}
\end{align}
Moreover, the Jacobi weight satisfies
\begin{equation*}
\left(1-t^2\right)\omega_{\alpha,\beta}(t)
=
\omega_{\alpha+1,\beta+1}(t).
\end{equation*}
We denote by $\widehat P_j^{(\alpha,\beta)}$ the orthonormal
Jacobi polynomial of degree $j$, namely
\begin{equation}\label{ortpol}
\widehat P_j^{(\alpha,\beta)}(x)
=
\frac{P_j^{(\alpha,\beta)}(x)}{\sqrt{K_j}},
\end{equation}
where
\[
K_j=
\int_{-1}^1
\left|P_j^{(\alpha,\beta)}(x)\right|^2
\omega_{\alpha,\beta}(x)dx.
\]
The classical formula for the Jacobi norm gives
\begin{equation}\label{Kj}
K_j
=
\frac{2^{\alpha+\beta+1}}{2j+\alpha+\beta+1}
\frac{\Gamma(j+\alpha+1)\Gamma(j+\beta+1)}
{\Gamma(j+1)\Gamma(j+\alpha+\beta+1)},
\quad j\ge0.
\end{equation}
We shall use the following standard consequence of the asymptotic expansion of
the Gamma function~\cite{Erdelyi:1953:HTF}. We include the proof for
completeness.
\begin{lemma}
\label{lem:gamma_ratio_asymptotic}
Let $a,b\in\mathbb{R}$ be fixed. Then
\[
\lim_{n\to\infty}\frac{\Gamma(n+a)}{\Gamma(n+b)}n^{b-a}=1.
\]
\end{lemma}

\begin{proof}
By Stirling's formula, we have
\[
\lim_{x\to+\infty}
\frac{\Gamma(x)}{\sqrt{2\pi}x^{x-\frac{1}{2}}e^{-x}}=1.
\]
Then, we get
\begin{eqnarray*}
    \lim_{n\to\infty}
\frac{\Gamma(n+a)}
{\sqrt{2\pi}(n+a)^{n+a-\frac{1}{2}}e^{-(n+a)}}&=&1,\\
\lim_{n\to\infty}
\frac{\Gamma(n+b)}
{\sqrt{2\pi}(n+b)^{n+b-\frac{1}{2}}e^{-(n+b)}}&=&1.
\end{eqnarray*}
Taking the ratio yields
\begin{equation}\label{stirl_formula}
    \lim_{n\to\infty}
\frac{\Gamma(n+a)}{\Gamma(n+b)}
e^{a-b}
\frac{(n+b)^{n+b-\frac{1}{2}}}{(n+a)^{n+a-\frac{1}{2}}}
=1.
\end{equation}
Moreover
\[
\frac{(n+b)^{n+b-\frac{1}{2}}}{(n+a)^{n+a-\frac{1}{2}}}
=
n^{b-a}
\left(1+\frac{b}{n}\right)^n
\left(1+\frac{b}{n}\right)^{b-\frac{1}{2}}
\left(1+\frac{a}{n}\right)^{-n}
\left(1+\frac{a}{n}\right)^{-a+\frac{1}{2}},
\]
and hence, using the elementary limits
\[
\left(1+\frac{b}{n}\right)^n\to e^b,
\quad
\left(1+\frac{a}{n}\right)^n\to e^a,
\]
\[
\left(1+\frac{b}{n}\right)^{b-\frac{1}{2}}\to 1,
\quad
\left(1+\frac{a}{n}\right)^{-a+\frac{1}{2}}\to 1,
\]
we obtain
\begin{equation}\label{lim_form}
    \lim_{n\to\infty}
e^{a-b}\frac{(n+b)^{n+b-\frac{1}{2}}}{(n+a)^{n+a-\frac{1}{2}}}n^{a-b}
=1.
\end{equation}
Combining~\eqref{stirl_formula} and~\eqref{lim_form}, we deduce
\[
\lim_{n\to\infty}\frac{\Gamma(n+a)}{\Gamma(n+b)}n^{b-a}=1,
\]
as claimed.
\end{proof}

Using  Lemma~\ref{lem:gamma_ratio_asymptotic} we derive the estimate on
the Jacobi norms. Since $\alpha$ and $\beta$ are fixed, the lemma gives
\[
\lim_{j\to\infty}
\frac{\Gamma(j+\alpha+1)}
{\Gamma(j+1)(j+1)^\alpha}
=1,
\quad
\lim_{j\to\infty}
\frac{\Gamma(j+\beta+1)}
{\Gamma(j+\alpha+\beta+1)(j+1)^{-\alpha}}
=1.
\]
Consequently, we find
\[
\lim_{j\to\infty}
\frac{\Gamma(j+\alpha+1)\Gamma(j+\beta+1)}
{\Gamma(j+1)\Gamma(j+\alpha+\beta+1)}
=1.
\]
Combining the latter with~\eqref{Kj}, we obtain
\[
K_j\sim
\frac{2^{\alpha+\beta+1}}{2j+\alpha+\beta+1},
\quad j\to\infty.
\]
In particular, there exist constants
$c_1(\alpha,\beta)>0$ and $C_1(\alpha,\beta)>0$ such that
\begin{equation}\label{eq:Kj_bound}
\frac{c_1(\alpha,\beta)}{j+1}
\le
K_j
\le
\frac{C_1(\alpha,\beta)}{j+1},
\quad j\ge0.
\end{equation}

The following weighted uniform estimate for Jacobi polynomials is a key ingredient for our analysis.

\begin{lemma}\label{lem:weighted_jacobi_uniform_bound}
Assume that $\alpha,\beta\ge \frac{1}{2}$. Then there exists a constant
$C(\alpha,\beta)>0$ such that
\[
\sup_{x\in[-1,1]}
\left|P_j^{(\alpha,\beta)}(x)\omega_{\alpha,\beta}(x)\right|^2
\le
\frac{C(\alpha,\beta)}{j+1},
\quad j\ge0.
\]
\end{lemma}

\begin{proof}
Let
\[
\left\{\widehat P_j^{(\alpha,\beta)}\right\}_{j\ge0}
\]
be the orthonormal Jacobi polynomials defined in~\eqref{ortpol}. By the
uniform estimate for orthonormal Jacobi polynomials proved
in~\cite{Nevai:1994:GJW}, there exists a constant
$C_0(\alpha,\beta)>0$ such that
\begin{equation*}
(1-x)^{\alpha+\frac{1}{2}}
(1+x)^{\beta+\frac{1}{2}}
\left|\widehat P_j^{(\alpha,\beta)}(x)\right|^2
\le C_0(\alpha,\beta),
\quad x\in[-1,1],\quad j\ge0.
\end{equation*}
Consequently, by~\eqref{ortpol}, we have
\begin{align}
\left|P_j^{(\alpha,\beta)}(x)\omega_{\alpha,\beta}(x)\right|^2
&=
K_j
\left|\widehat P_j^{(\alpha,\beta)}(x)\right|^2
(1-x)^{2\alpha}(1+x)^{2\beta} \notag\\
&\le
K_j C_0(\alpha,\beta)
(1-x)^{\alpha-\frac{1}{2}}
(1+x)^{\beta-\frac{1}{2}}.
\label{eq:weighted_poly_bound}
\end{align}
Since $\alpha,\beta\ge \frac{1}{2}$, the last factor is bounded on $[-1,1]$.
Thus
\[
\sup_{x\in[-1,1]}
\left|P_j^{(\alpha,\beta)}(x)\omega_{\alpha,\beta}(x)\right|^2
\le
C_1(\alpha,\beta)K_j.
\]
Using~\eqref{eq:Kj_bound} and renaming the constant, we obtain
\[
\sup_{x\in[-1,1]}
\left|P_j^{(\alpha,\beta)}(x)\omega_{\alpha,\beta}(x)\right|^2
\le
\frac{C(\alpha,\beta)}{j+1},
\quad j\ge0.
\]
\end{proof}

\subsection{The weighted histopolation problem}

Let $X_N=\left\{x_0,\dots,x_N\right\}\subset[-1,1]$ be such that
\begin{equation}\label{nodecond}
    -1=x_0<\dots<x_N=1
\end{equation}
and set
\[
s_i=\left[x_{i-1},x_{i}\right], \quad h_i=x_i-x_{i-1}, \quad i=1,\dots,N.
\]
Given $f\in L^1\left([-1,1],\omega_{\alpha,\beta}\right)$, we define the weighted cell averages
\begin{equation}
b_i
=
\frac{1}{h_i}
\int_{s_i} f(t)\omega_{\alpha,\beta}(t)dt,
\quad i=1,\dots,N.
\end{equation}
Let $\left\{\varphi_0,\dots,\varphi_{N-1}\right\}$ be a basis of
$\mathbb{P}_{N-1}([-1,1])$. The weighted histopolation problem consists in
determining a polynomial
\begin{equation}\label{pol_hist}
p_{N-1}(t)=\sum_{j=1}^N c_j\varphi_{j-1}(t)
\end{equation}
such that
\begin{equation*}
\frac{1}{h_i}\int_{s_i} p_{N-1}(t)\omega_{\alpha,\beta}(t)dt=b_i,
\quad i=1,\dots,N.
\end{equation*}
Equivalently, the vector of coefficients $\boldsymbol{c}=\left[c_1,\dots,c_N\right]^\top$ is determined by the
linear system
\begin{equation*}
H_N \boldsymbol{c}=\boldsymbol{b},
\end{equation*}
where $\boldsymbol{b}=\left[b_1,\dots,b_N\right]^\top$, and
\begin{equation}\label{eq:matHNentries}
    \left[H_N\right]_{i,j}
=
\frac{1}{h_i}
\int_{s_i}
\varphi_{j-1}(t)\omega_{\alpha,\beta}(t)dt,
\quad i,j=1,\dots,N.
\end{equation}

It is well known that, in the classical unweighted case, histopolation on a
family of subintervals with pairwise intersections of measure zero is
unisolvent; see, e.g.,~\cite{Bruno:2024:PIO,Bruno:2025:OTC}. The next result shows that the same property remains valid in the weighted setting.

\begin{theorem}
Let $X_N=\left\{x_0,\dots,x_N\right\}$ satisfy~\eqref{nodecond}, and let
\[
s_i=\left[x_{i-1},x_i\right], \quad h_i=x_i-x_{i-1}, \quad i=1,\dots,N.
\]
Let $\omega\in L^1(-1,1)$ be such that $\omega(t)>0$ a.e. on $(-1,1)$.
Then, for any data
\[
\eta_1,\dots,\eta_N\in\mathbb{R},
\]
there exists a unique
polynomial $p_{N-1}\in\mathbb{P}_{N-1}$ such that
\[
\frac{1}{h_i}\int_{s_i} p_{N-1}(t)\omega(t)dt = \eta_i,
\quad i=1,\dots,N.
\]
Equivalently, for any basis $\left\{\varphi_k\right\}_{k=0}^{N-1}$ of $\mathbb{P}_{N-1}$, the associated weighted histopolation matrix is
nonsingular.
\end{theorem}
\begin{proof}
We consider the linear map
\[
\Phi:p\in\mathbb{P}_{N-1}\mapsto
\left[
\frac{1}{h_1}\int_{s_1}p(t)\omega(t)dt,
\dots,
\frac{1}{h_N}\int_{s_N}p(t)\omega(t)dt
\right]^\top
\in\mathbb{R}^N .
\]
With respect to the basis $\left\{\varphi_k\right\}_{k=0}^{N-1}$ of
$\mathbb{P}_{N-1}$ and the canonical basis of $\mathbb{R}^N$, the matrix
representation of $\Phi$ coincides with the associated weighted histopolation
matrix. Since
\[
\dim\left(\mathbb{P}_{N-1}\right)=\dim\left(\mathbb{R}^N\right)=N,
\]
the map $\Phi$ is an isomorphism if and only if it is injective. Therefore,
to prove the theorem it is enough to prove that $\Phi$ is injective.

Let $p\in\mathbb{P}_{N-1}$ satisfy
\begin{equation}\label{nullphi}
[\Phi(p)]_i=0, \quad i=1,\dots,N .
\end{equation}
For each $i=1,\dots,N$, the continuity of $p$ on $s_i$ implies that $p$
attains its minimum and maximum on $s_i$. We denote them by
\[
m_i=\min_{t\in s_i}p(t), \quad M_i=\max_{t\in s_i}p(t).
\]
Moreover, since $\omega(t)>0$ a.e. on $(-1,1)$ and $h_i>0$, we have
\[
A_i=\frac{1}{h_i}\int_{s_i}\omega(t)dt>0 .
\]
From
\[
m_i\le p(t)\le M_i, \quad t\in s_i,
\]
and $\omega(t)>0$ a.e., we obtain
\[
m_i A_i \le [\Phi(p)]_i \le M_i A_i,
\quad i=1,\dots,N.
\]
By~\eqref{nullphi}, and since $A_i>0$, it follows that
\[
m_i\le 0\le M_i,
\quad i=1,\dots,N.
\]
If, for some $i$, either $m_i=0$ or $M_i=0$, then $p$ has a constant sign on
$s_i$. In this case, the identity
\[
\int_{s_i}p(t)\omega(t)dt=0
\]
and the positivity of $\omega$ imply that $p=0$ on $s_i$. Hence, since a
polynomial that vanishes on an interval is identically zero, we conclude that
$p=0$ on $[-1,1]$.

Otherwise, for any $i=1,\dots,N$, we have
\[
m_i<0<M_i.
\]
Thus $p$ changes sign on each interval $s_i$. By the intermediate value
theorem, for any $i=1,\dots,N$ there exists
\[
\xi_i\in\left(x_{i-1},x_i\right)
\]
such that $p(\xi_i)=0$. The points $\xi_1,\dots,\xi_N$ are distinct, and hence
$p$ has at least $N$ distinct zeros. Since $\deg(p)\le N-1$, it follows that
$p=0$.

This proves that the only element $p$ satisfying~\eqref{nullphi} is the zero
polynomial. Thus $\Phi$ is injective.
\end{proof}

For the subsequent analysis, we consider the Jacobi basis
\begin{equation}\label{jacbasiabov}
\varphi_{j-1}(t)=P^{(\alpha+1,\beta+1)}_{j-1}(t),
\quad j=1,\dots,N.
\end{equation}
Associated with this basis, we introduce the weighted primitives
\begin{equation}\label{eq:primitive}
\psi_j(x)
=
\int_{-1}^{x}
P^{(\alpha+1,\beta+1)}_{j-1}(t)
\omega_{\alpha,\beta}(t)dt, \quad j=1,\dots,N.
\end{equation}
By~\eqref{eq:jacobi-derivative}, these functions can equivalently be written as
\begin{equation}\label{eq:primitive-derivative}
\psi_j(x)
=
\frac{2}{j+\alpha+\beta+1}
\int_{-1}^x
\left(P_j^{(\alpha,\beta)}(t)\right)'
\omega_{\alpha,\beta}(t)dt,\quad j=1,\dots,N.
\end{equation}
We next introduce two discrete operators which are used to factorize the
weighted histopolation matrix. The first one is the backward difference
operator $\Delta_N\in\mathbb{R}^{N\times(N+1)}$, defined by
\begin{equation}\label{mat:DeltaN}
\left[\Delta_N \boldsymbol{u}\right]_i = \frac{u_i-u_{i-1}}{h_i}, \quad i=1,\dots,N, \quad \boldsymbol{u}=\left[u_0,\dots,u_N\right]^\top\in\mathbb{R}^{N+1}.
\end{equation}
The second one
is the primitive sampling matrix $\Psi_N\in\mathbb{R}^{(N+1)\times N}$, defined by
\begin{equation}\label{mat:PsiN}
\left[\Psi_N\right]_{k,j}
=
\psi_j\left(x_k\right), \quad k=0,\dots,N, \quad j=1,\dots,N.
\end{equation}
The next result provides an explicit expression for $\psi_j$, which is employed in the spectral analysis.

\begin{lemma}\label{lem:IBP}
Let $\alpha,\beta>0$. Then, for any $j=1,\dots,N$ and $x\in[-1,1]$, we have
\begin{equation}\label{eq:IBP-psi}
\psi_j(x)
=
\frac{2\left(
P_j^{(\alpha,\beta)}(x)\omega_{\alpha,\beta}(x)
+(\alpha-\beta)I_j(x)
+(\alpha+\beta)J_j(x)
\right)}{j+\alpha+\beta+1},
\end{equation}
where
\begin{align}
I_j(x)
&=
\int_{-1}^{x} P_j^{(\alpha,\beta)}(t)
\omega_{\alpha-1,\beta-1}(t)dt, \label{eq:Ij}\\
J_j(x)
&=
\int_{-1}^{x} tP_j^{(\alpha,\beta)}(t)
\omega_{\alpha-1,\beta-1}(t)dt. \label{eq:Jj}
\end{align}
\end{lemma}

\begin{proof}
From~\eqref{eq:primitive-derivative} and integration by parts, we find
\begin{eqnarray} \notag
&&\psi_j(x)
= \frac{2}{ j+\alpha+\beta+1}
 \int_{-1}^{x} \left(P_j^{(\alpha,\beta)}(t)\right)'\omega_{\alpha,\beta}(t)dt
\\&=& \frac{2}{ j+\alpha+\beta+1}\left(
P_j^{(\alpha,\beta)}(x)\omega_{\alpha,\beta}(x)
-
\int_{-1}^{x}
P_j^{(\alpha,\beta)}(t)
\frac{d}{dt}\omega_{\alpha,\beta}(t)dt\right), \label{eq:primitivePsi}
\end{eqnarray}
where we have used $\omega_{\alpha,\beta}(-1)=0$, since $\beta>0$.
A direct computation gives
\begin{equation*}
\frac{d}{dt}\omega_{\alpha,\beta}(t)
=
\left[(\beta-\alpha)-(\alpha+\beta)t\right]
\omega_{\alpha-1,\beta-1}(t).
\end{equation*}
Hence
\begin{eqnarray}\notag
    && \int_{-1}^{x}
P_j^{(\alpha,\beta)}(t)
\frac{d}{dt}\omega_{\alpha,\beta}(t)dt=\\ \notag
&=&  (\beta-\alpha)\int_{-1}^{x}
P_j^{(\alpha,\beta)}(t)
\omega_{\alpha-1,\beta-1}(t)dt-(\alpha+\beta)\int_{-1}^{x}t
P_j^{(\alpha,\beta)}(t)
\omega_{\alpha-1,\beta-1}(t)dt\\ \label{eq:rel_primitive_IjJj}
&=&(\beta-\alpha)I_j(x)
-
(\alpha+\beta)J_j(x),
\end{eqnarray}
where $I_j$ and $J_j$ are defined in~\eqref{eq:Ij} and~\eqref{eq:Jj}, respectively.
Substituting~\eqref{eq:rel_primitive_IjJj} into~\eqref{eq:primitivePsi} leads to \eqref{eq:IBP-psi}.
\end{proof}

The next lemma shows that the three-term recurrence relation~\eqref{eq:jacobi-recurrence}
expresses each $J_j$ as a tridiagonal combination of the weighted
primitives $I_{j+1}$, $I_j$, and $I_{j-1}$.

\begin{lemma}\label{lem:spectral-locality}
Let $\alpha,\beta>0$. Then, for any $j=1,\dots,N$ and $x\in[-1,1]$, the following equality holds
\[
J_j(x)
=
a_j I_{j+1}(x)
+
b_j I_j(x)
+
c_j I_{j-1}(x),
\]
where
\[
I_k(x)=\int_{-1}^{x} P_k^{(\alpha,\beta)}(t)
\omega_{\alpha-1,\beta-1}(t)dt,\quad k\ge0,
\]
and $a_j,b_j,c_j$ are defined in~\eqref{aj},~\eqref{bj}, and~\eqref{cj},
respectively.
\end{lemma}

\begin{proof}
Substituting~\eqref{eq:jacobi-recurrence} into~\eqref{eq:Jj} and using
linearity, we obtain
\begin{eqnarray*}
J_j(x)
&=&
\int_{-1}^{x}
\left(
a_j P_{j+1}^{(\alpha,\beta)}(t)
+
b_j P_j^{(\alpha,\beta)}(t)
+
c_j P_{j-1}^{(\alpha,\beta)}(t)
\right)
\omega_{\alpha-1,\beta-1}(t)dt\\
&=&
a_j \int_{-1}^{x} P_{j+1}^{(\alpha,\beta)}(t)
\omega_{\alpha-1,\beta-1}(t)dt
+
b_j \int_{-1}^{x} P_j^{(\alpha,\beta)}(t)
\omega_{\alpha-1,\beta-1}(t)dt\\
&+&
c_j \int_{-1}^{x} P_{j-1}^{(\alpha,\beta)}(t)
\omega_{\alpha-1,\beta-1}(t)dt.
\end{eqnarray*}
The conclusion follows from the definitions of $I_{j+1}(x)$, $I_j(x)$, and
$I_{j-1}(x)$.
\end{proof}

\begin{remark}
By Lemma~\ref{lem:IBP} and Lemma~\ref{lem:spectral-locality}, for any
$j=1,\dots,N$, we obtain
\begin{equation}\label{eq:psiI_j}
\psi_j(x)
=
\widetilde r_j(x)
+
u_j I_{j+1}(x)
+
d_j I_j(x)
+
\ell_j I_{j-1}(x),
\quad x\in[-1,1],
\end{equation}
where
\[
\widetilde r_j(x)
=
\frac{2}{j+\alpha+\beta+1}
P_j^{(\alpha,\beta)}(x)\omega_{\alpha,\beta}(x),
\]
and
\begin{equation}\label{ujdjlj}
u_j=
\frac{2(\alpha+\beta)a_j}{j+\alpha+\beta+1},
\
d_j=
\frac{2\left((\alpha-\beta)+(\alpha+\beta)b_j\right)}{j+\alpha+\beta+1},
\
\ell_j=
\frac{2(\alpha+\beta)c_j}{j+\alpha+\beta+1}.
\end{equation}
\end{remark}

For the  Jacobi basis~\eqref{jacbasiabov}, we now show that the weighted histopolation
matrix admits an exact factorization in terms of the backward difference
operator and the primitive sampling matrix.
\begin{theorem}\label{thm:exact-factorization}
Let $\alpha,\beta>0$ and let $X_N=\left\{x_0,\dots,x_N\right\}$ satisfy~\eqref{nodecond}.
Then the weighted Jacobi histopolation matrix $H_N$ associated with the basis~\eqref{jacbasiabov}
admits the factorization
\[
H_N=\Delta_N\Psi_N,
\]
where $\Delta_N$ and $\Psi_N$ are defined in~\eqref{mat:DeltaN} and
\eqref{mat:PsiN}, respectively.
\end{theorem}
\begin{proof}
Let $\boldsymbol{c}=\left[c_1,\dots,c_N\right]^\top\in\mathbb{R}^N$. Consider the polynomial
\[
p_{N-1}=\sum_{j=1}^{N}c_jP_{j-1}^{(\alpha+1,\beta+1)}\in\mathbb{P}_{N-1}.
\]
Then, by~\eqref{eq:matHNentries}, we have
\begin{equation*}
\left[H_N \boldsymbol{c}\right]_i=\sum_{j=1}^N \left[H_N\right]_{i,j}c_j =\frac{1}{h_i} \sum_{j=1}^{N} c_j
\int_{s_i}
P_{j-1}^{(\alpha+1,\beta+1)}(t)
\omega_{\alpha,\beta}(t)dt,
\end{equation*}
for any $i=1,\dots,N.$
Using~\eqref{eq:primitive}, we obtain
\[
\int_{s_i}
P_{j-1}^{(\alpha+1,\beta+1)}(t)
\omega_{\alpha,\beta}(t)dt
=
\psi_j\left(x_i\right)-\psi_j\left(x_{i-1}\right), \quad i=1,\dots,N.
\]
Hence
\[
\left[H_N \boldsymbol{c}\right]_i
=
\sum_{j=1}^{N} c_j\frac{\psi_j\left(x_i\right)-\psi_j\left(x_{i-1}\right)}{h_i}, \quad i=1,\dots,N.
\]
Thus, by the definitions~\eqref{mat:DeltaN} and~\eqref{mat:PsiN}, we have
\[
\left[H_N\boldsymbol{c}\right]_i=\left[\Delta_N\Psi_N\boldsymbol{c}\right]_i, \quad i=1,\dots,N.
\]
Since $\boldsymbol{c}$ is arbitrary, it follows that
\[
H_N=\Delta_N\Psi_N.
\]
\end{proof}

We now describe the algebraic structure of the primitive sampling matrix.
To this end, we introduce the following auxiliary matrices. Let
$I_N^{\mathrm{ext}}\in\mathbb{R}^{(N+1)\times (N+2)}$ be defined by
\begin{equation}\label{INext}
\left[I_N^{\mathrm{ext}}\right]_{k+1,r+1}=I_r\left(x_k\right),
\quad k=0,\dots,N,\quad r=0,\dots,N+1,
\end{equation}
where $I_r$ is defined in~\eqref{eq:Ij}. We also define $R_N\in\mathbb{R}^{(N+1)\times N}$ by
\begin{equation}\label{matRN}
    \left[R_N\right]_{k+1,j}
=
\widetilde r_j(x_k),
\quad k=0,\dots,N,\quad j=1,\dots,N,
\end{equation}
with
\begin{equation}\label{eq:rjtilde}
    \widetilde r_j(x)=
\frac{2}{j+\alpha+\beta+1}
P_j^{(\alpha,\beta)}(x)\omega_{\alpha,\beta}(x).
\end{equation}
Finally, let $T_N^{(J)}\in\mathbb{R}^{(N+2)\times N}$ be the matrix whose only
nonzero entries are
\begin{equation}\label{matrixTNj}
    \left[T_N^{(J)}\right]_{j,j}=\ell_j,
\quad
\left[T_N^{(J)}\right]_{j+1,j}=d_j,
\quad
\left[T_N^{(J)}\right]_{j+2,j}=u_j,
\quad j=1,\dots,N,
\end{equation}
where
\[
u_j=
\frac{2(\alpha+\beta)a_j}{j+\alpha+\beta+1},
\quad
d_j=
\frac{2\left((\alpha-\beta)+(\alpha+\beta)b_j\right)}{j+\alpha+\beta+1},
\quad
\ell_j=
\frac{2(\alpha+\beta)c_j}{j+\alpha+\beta+1}.
\]

\begin{theorem}\label{thm:struct-tridiag}
Let $\alpha,\beta>0$ and let $X_N=\left\{x_0,\dots,x_N\right\}$ satisfy~\eqref{nodecond}.
Then the primitive sampling matrix $\Psi_N$, defined in~\eqref{mat:PsiN}, admits
the decomposition
\begin{equation}\label{PsiN}
\Psi_N = R_N + I_N^{\mathrm{ext}}T_N^{(J)}.
\end{equation}
\end{theorem}

\begin{proof}
By~\eqref{eq:psiI_j}, for $j=1,\dots,N$, we have
\[
\psi_j(x)
=
\widetilde r_j(x)
+
u_j I_{j+1}(x)
+
d_j I_j(x)
+
\ell_j I_{j-1}(x),
\quad x\in[-1,1].
\]
Sampling this identity at the grid points $x_k$, $k=0,\dots,N$, and arranging
the resulting relations columnwise gives
\[
\Psi_N = R_N + I_N^{\mathrm{ext}}T_N^{(J)}.
\]
\end{proof}

\section{GLT notation and basic tools}\label{sec3}

In this section we recall the notation and the basic tools from the theory of
Generalized Locally Toeplitz (GLT) sequences that are used in the spectral
analysis of the histopolation matrix sequences. We only report the material needed in the sequel, and refer the reader
to~\cite{Garoni:2018:GLT1,Garoni:2018:GLT2} for a complete treatment.

For a matrix $A\in \mathbb{C}^{m\times n}$, we denote by
\[
\sigma_1(A)\ge \sigma_2(A)\ge \cdots \ge \sigma_{\min\{m,n\}}(A)\ge0
\]
its singular values, arranged in nonincreasing order. If $A\in\mathbb C^{n\times n}$ is a square matrix, we denote by
\[
\lambda_1(A),\ldots,\lambda_n(A)
\]
its eigenvalues. For $q\in[1,\infty]$, the Schatten $q$-norm of $A$ is denoted by
$\left\|A\right\|_{S,q}$; see \cite{Bhatia:1997:MAN}. In particular, the Frobenius norm $\left\|\cdot\right\|_{F}$ is the Schatten $2$-norm that is
\[
\left\|\cdot\right\|_{F}=\left\|\cdot\right\|_{S,2}.
\]
The spectral norm of $A$, namely $\left\|A\right\|_{S,\infty}=\sigma_1(A)$, is denoted by
$\left\|A\right\|$. If $[A]_{ij}=a_{ij}$, we also use
\[
\left\|A\right\|_1
=
\max_j\sum_i \left|a_{ij}\right|,
\quad
\left\|A\right\|_{\infty}
=
\max_i\sum_j \left|a_{ij}\right|,
\]
which are the $\ell^1\rightarrow \ell^1$ and $\ell^\infty\rightarrow \ell^\infty$ induced norms, respectively. We note that all the Schatten $q$-norms are unitarily invariant while the two induced norms above are not~\cite{Bhatia:1997:MAN}.
Moreover, we denote by $C_c\left(\mathbb K\right)$ the space of continuous functions
with compact support on $\mathbb K$, where $\mathbb K=\mathbb R$ or
$\mathbb K=\mathbb C$.

\begin{definition}
\label{def:sv-ev-distribution}
Let $\left\{A_N\right\}_N$ be a matrix sequence such that
\[
A_N\in\mathbb{C}^{m_N\times n_N},
\]
and
\[
d_N:=\min\left\{m_N,n_N\right\}, \quad \lim_{N\to \infty}d_N=\infty.
\]
Let
\[
\chi:D\subset \mathbb R^k\to\mathbb C
\]
be a measurable function defined on a measurable set $D$ with
\[
0<\mu_k\left(D\right)<\infty,
\]
where $\mu_k$ denotes the Lebesgue measure on $\mathbb R^k$.

We say that $\left\{A_N\right\}_N$ has asymptotic singular value distribution
described by $\chi$, and we write
\begin{equation*}
    \left\{A_N\right\}_N\sim_\sigma \chi,
\end{equation*}
if
\[
\lim_{N\to\infty}
\frac1{d_N}
\sum_{i=1}^{d_N}
F\left(\sigma_i\left(A_N\right)\right)
=
\frac{1}{\mu_k\left(D\right)}
\int_D F\left(\left|\chi\left(\boldsymbol x\right)\right|\right)
d\boldsymbol x,
\quad
\forall F\in C_c\left(\mathbb R\right).
\]
In this case, $\chi$ is called the singular value symbol of
$\left\{A_N\right\}_N$.

If, in addition, $m_N=n_N=d_N$ for any $N$, we say that
$\left\{A_N\right\}_N$ has asymptotic eigenvalue, or spectral, distribution
described by $\chi$, and we write
\begin{equation*}
    \left\{A_N\right\}_N\sim_\lambda \chi,
\end{equation*}
if
\[
\lim_{N\to\infty}
\frac{1}{d_N}
\sum_{i=1}^{d_N}
F\left(\lambda_i\left(A_N\right)\right)
=
\frac{1}{\mu_k\left(D\right)}
\int_D F\left(\chi\left(\boldsymbol x\right)\right)d\boldsymbol x,
\quad
\forall F\in C_c\left(\mathbb C\right).
\]
In this case, $\chi$ is called the spectral symbol of
$\left\{A_N\right\}_N$.

If both distributions hold with the same function $\chi$, then we write
\[
\left\{A_N\right\}_N \sim_{\sigma,\lambda} \chi.
\]
\end{definition}

\begin{remark}\label{rearrangements}
It is worth noticing that when the distribution exists it is far from unique as discussed in \cite{Barbarino:2022:CAT,Ekstrom:2018:EAE}. Of special interest are the monotone rearrangements (nondecreasing or nonincreasing), which are used e.g. to check the adherence of the singular value symbol and the ordered plots of the singular values and to check the adherence of the spectral symbol and the ordered plots of the eigenvalues; see the numerical test in Section \ref{sec5}. We also refer to \cite[Remark 2.7]{Mazza:2019:SAA} for other practical considerations regarding these notions.
\end{remark}

\begin{definition}
    A matrix sequence $\left\{Z_N\right\}_N$ is called zero-distributed if
\[
\left\{Z_N\right\}_N\sim_\sigma 0.
\]
\end{definition}

In the following, we shall use the standard characterization of
zero-distributed sequences, see~\cite[Th.~3.2]{Garoni:2018:GLT1}.
\begin{theorem}\label{thmGaroni}
Let $\left\{Z_N\right\}_N$ be a matrix sequence, with
$Z_N\in\mathbb C^{m_N\times n_N}$ and
\[
d_N=\min\left\{m_N,n_N\right\}, \quad \lim_{N\to\infty}d_N=\infty.
\]
Then the following conditions are
equivalent:
 \begin{itemize}
     \item[1)] $\left\{Z_N\right\}_N\sim_\sigma0$.
     \item[2)] For any $\varepsilon>0$, we have
     \[
     \lim_{N\to \infty} \frac{\#\left\{j\in\{1,\dots,d_N\} \, :\, \sigma_j\left(Z_N\right)>\varepsilon\right\}}{d_N}=0.
     \]
     \item[3)] There exist two matrix sequences $\left\{R_N\right\}_N$ and
$\left\{E_N\right\}_N$ such that
\[
Z_N=R_N+E_N,
\]
and
     \[
    \lim_{N\to \infty} \frac{\operatorname{rank}\left(R_N\right)}{d_N}=0, \quad \lim_{N\to\infty} \left\|E_N\right\|=0.
     \]
 \end{itemize}
\end{theorem}

Let $f\in L^1\left([-\pi,\pi]\right)$. The $N$-th Toeplitz matrix generated by
$f$ is defined as
\begin{equation}\label{toe:def}
T_N\left(f\right)
=
\left[\widehat f_{i-j}\right]_{i,j=1}^N
\in\mathbb C^{N\times N},
\end{equation}
where
\begin{equation}\label{fou:def}
\widehat f_k=
\frac1{2\pi}
\int_{-\pi}^{\pi}
f\left(\theta\right)e^{-ik\theta}d\theta,
\quad k\in\mathbb Z.
\end{equation}
The function $f$ is called the generating function of the Toeplitz sequence
$\left\{T_N\left(f\right)\right\}_N$. In particular, if
\begin{equation}\label{toe:def-pol}
p(\theta)=\sum_{r=r_1}^{r_2}p_r e^{ir\theta},
\quad r_1,r_2\in\mathbb Z,
\end{equation}
is a Laurent polynomial in the variable $e^{i\theta}$, then $T_N(p)$ is banded.
Indeed,
\begin{equation}\label{fou:def-pol}
\widehat p_k=
\begin{cases}
p_k, & k\in\{r_1,\ldots,r_2\},\\
0, & k\notin\{r_1,\ldots,r_2\}.
\end{cases}
\end{equation}
With this convention, the elementary symbol $e^{ir\theta}$ gives the Toeplitz
shift supported on the diagonal $i-j=r$.

Let $a:[0,1]\to\mathbb C$ be a Riemann-integrable function. We define the
diagonal sampling matrix
\[
D_N(a)
=
\operatorname{diag}\left[
a\left(\frac{1}{N}\right),\ldots,
a\left(\frac{N}{N}\right)
\right].
\]
The sequence $\left\{D_N(a)\right\}_N$ is called the diagonal sampling sequence
associated with $a$. We shall also use the standard extension of diagonal sampling sequences to
measurable functions that are finite a.e. on $[0,1]$. Thus, for instance, the
sequence
\[
\operatorname{diag}\left(\frac{N}{1},\frac{N}{2},\ldots,\frac{N}{N}\right)
\]
is interpreted as a diagonal sampling sequence with symbol $y^{-1}$ on
$(0,1]$. The behaviour at $y=0$ is irrelevant in the GLT setting, since GLT symbols belong to the $*$-algebra of measurable functions over $[0,1]\times [-\pi,\pi]$ and hence they are identified up to equality a.e.

\subsection{GLT sequences}

We now recall the GLT results needed for the asymptotic analysis. Since the
matrices considered in this paper arise from one-dimensional discretizations,
we work in the scalar unilevel setting and all the involved matrices have size $N\times N$. In this framework, a GLT sequence is a
matrix sequence $\left\{A_N\right\}_N$ associated with a measurable function
\[
\kappa:[0,1]\times[-\pi,\pi]\to\mathbb C,
\]
called its GLT symbol. We write
\[
\left\{A_N\right\}_N\sim_{\mathrm{GLT}}\kappa(y,\theta)
\]
to indicate that $\left\{A_N\right\}_N$ is a GLT sequence with symbol $\kappa$.
The symbol is uniquely determined up to equality a.e. on
$[0,1]\times[-\pi,\pi]$. The variable $y\in[0,1]$ records the dependence on the normalized index $i/N$, as in
diagonal sampling sequences, while $\theta\in[-\pi,\pi]$ is the Fourier variable
associated with Toeplitz sequences.

The GLT class can be described as the smallest class of matrix sequences which
contains Toeplitz sequences, diagonal sampling sequences, and zero-distributed
sequences, and which is closed under algebraic operations and under limits in the sense of
approximating classes of sequences (a.c.s.).

We write
\[
\left\{B_{N,m}\right\}_N
\xrightarrow{\mathrm{a.c.s.}}
\left\{A_N\right\}_N, \ \ m\rightarrow \infty,
\]
if, for every $m$, and for any sufficiently large $N$, there exist matrices
$R_{N,m}$ and $E_{N,m}$ such that
\[
A_N=B_{N,m}+R_{N,m}+E_{N,m},
\]
with
\[
\operatorname{rank}\left(R_{N,m}\right)\le c_m N,
\quad
\left\|E_{N,m}\right\|\le s_m,
\]
where
\[
\lim_{m\to\infty}c_m=0,
\quad
\lim_{m\to\infty}s_m=0.
\]
In other words, $\left\{B_{N,m}\right\}_N$ approximates
$\left\{A_N\right\}_N$ up to a perturbation by matrices of asymptotically
negligible rank and a perturbation by matrices with vanishing spectral norm.

\begin{remark}\label{link-appr-theory}
It is worth observing that the a.c.s. convergence is associated to a rigorous topology for matrix sequences. The following two facts are the basic bricks:
\begin{description}
\item[A1)] $\left\{B_{N,m}\right\}_N \xrightarrow{\mathrm{a.c.s.}}
\left\{A_N\right\}_N$, $\left\{B_{N,m}\right\}_N \sim_\sigma \psi_m$,
$\psi_m\rightarrow \phi$ a.e. imply
$\left\{A_N\right\}_N \sim_\sigma \phi$;
\item[A2)] $\left\{B_{N,m}\right\}_N \xrightarrow{\mathrm{a.c.s.}}
\left\{A_N\right\}_N$, $\left\{B_{N,m}\right\}_N \sim_\lambda \psi_m$,
$\psi_m\rightarrow \phi$ a.e., and all the involved matrices are Hermitian
for all sufficiently large $N$, imply
$\left\{A_N\right\}_N \sim_\lambda \phi$.
\end{description}
We notice that items A1) and A2) identify the foundation of an approximation theory for spectral and singular value distributions of matrix sequences; in regard to A2) see also the last part of axiom \textbf{GLT 5}. In this respect, the whole GLT analysis is built on these approximation tools.
\end{remark}

The GLT class satisfies several algebraic and topological properties, which
are treated in detail in~\cite{Serra:2001:DRO,Serra:2003:GLT,Serra:2006:TGC,Tilli:1998:LTS,Garoni:2018:GLT1,Garoni:2018:GLT2,Barbarino:2020:BGL2,Barbarino:2020:BGL1}.
In this paper, we only need the scalar unilevel setting, and we therefore
recall below the main operational rules used in the sequel. In this setting,
these rules are equivalent to the constructive definition of GLT sequences
given in~\cite{Serra:2003:GLT}.

\begin{enumerate}

\item[\textbf{GLT 1}] If
\[
\left\{A_N\right\}_N\sim_{\mathrm{GLT}}\kappa,
\]
then
\[
\left\{A_N\right\}_N\sim_\sigma \kappa
\]
in the sense of Definition~\ref{def:sv-ev-distribution}, with $k=2$ and
\[
D=[0,1]\times[-\pi,\pi].
\]
If, in addition, every $A_N$ is Hermitian, then
\[
\left\{A_N\right\}_N\sim_\lambda \kappa.
\]

\item[\textbf{GLT 2}] The basic generators satisfy
\[
\left\{T_N\left(f\right)\right\}_N
\sim_{\mathrm{GLT}} f\left(\theta\right),
\quad
\left\{D_N\left(a\right)\right\}_N
\sim_{\mathrm{GLT}} a\left(y\right),
\]
for any $f\in L^1\left([-\pi,\pi]\right)$ and any Riemann integrable function $a$ defined on $[0,1]$. Moreover
\[
\left\{Z_N\right\}_N\sim_{\mathrm{GLT}}0
\quad\Longleftrightarrow\quad
\left\{Z_N\right\}_N\sim_\sigma 0.
\]

\item[\textbf{GLT 3}] If
\[
\left\{A_N\right\}_N\sim_{\mathrm{GLT}}\kappa,
\quad
\left\{B_N\right\}_N\sim_{\mathrm{GLT}}\varsigma,
\]
then
\begin{itemize}
    \item \[\left\{A_N^*\right\}_N\sim_{\mathrm{GLT}}\overline{\kappa},\]
    \item
    \[
\left\{\alpha A_N+\beta B_N\right\}_N
\sim_{\mathrm{GLT}}
\alpha\kappa+\beta\varsigma, \quad \alpha,\beta\in\mathbb C,
    \]
\item
\[
\left\{A_NB_N\right\}_N
\sim_{\mathrm{GLT}}
\kappa\varsigma.
\]
\end{itemize}
Moreover, if $\kappa$ is nonzero a.e., then
\begin{itemize}
    \item
\[
\left\{A_N^\dagger\right\}_N
\sim_{\mathrm{GLT}}
\kappa^{-1},
\]
where $A_N^\dagger$ denotes the Moore--Penrose pseudo-inverse of $A_N$.
\end{itemize}

\item[\textbf{GLT 4}] Let
\[
\left\{B_{N,m}\right\}_N\sim_{\mathrm{GLT}}\kappa_m
\]
for every $m$. If
\[
\left\{B_{N,m}\right\}_N
\xrightarrow{\mathrm{a.c.s.}}
\left\{A_N\right\}_N
\]
and
\[
\kappa_m\to\kappa
\]
in measure on $[0,1]\times[-\pi,\pi]$, then
\[
\left\{A_N\right\}_N\sim_{\mathrm{GLT}}\kappa.
\]

\item[\textbf{GLT 5}] Suppose that
\[
\left\{A_N\right\}_N\sim_{\mathrm{GLT}}\kappa
\]
and that
\[
A_N=X_N+Y_N,
\]
where every $X_N$ is Hermitian,
\[
\left\|X_N\right\|\le C,
\quad
\left\|Y_N\right\|\le C,
\]
for a constant $C$ independent of $N$, and
\[
\lim_{N\to \infty}\frac{\left\|Y_N\right\|_{S,1}}{N}=0.
\]
Then
\[
\left\{A_N\right\}_N\sim_\lambda \kappa.
\]
Sequences satisfying these assumptions are called quasi-Hermitian.

\item[\textbf{GLT 6}] If
\[
\left\{A_N\right\}_N\sim_{\mathrm{GLT}}\kappa
\]
and every $A_N$ is Hermitian, then, for every continuous function
$f:\mathbb C\to\mathbb C$,
\[
\left\{f\left(A_N\right)\right\}_N
\sim_{\mathrm{GLT}}
f\left(\kappa\right).
\]
\end{enumerate}
As an immediate consequence, zero-distributed perturbations do not change the
singular value symbol of a GLT sequence. Indeed, if
\[
\left\{A_N\right\}_N\sim_{\mathrm{GLT}}\kappa
\quad\text{and}\quad
\left\{Z_N\right\}_N\sim_\sigma 0,
\]
then
\[
\left\{A_N+Z_N\right\}_N\sim_\sigma \kappa.
\]

Further tools that we use in the subsequent derivations are the notions of sparsely vanishing and sparsely unbounded matrix sequences.

\begin{definition}\label{s.u.}
A matrix sequence $\{A_N\}$ is said to be sparsely unbounded, if for every $M>0$ there exists $n_M$ such that, for $N\ge n_M$,
$$\frac{\#\{i\in\{1,\ldots,N\}:\ \sigma_i(A_N)>M\}}{N}\le r(M),$$
where $\lim_{M\to\infty}r(M)=0$.
\end{definition}

\begin{definition}\label{s.v.}
A matrix sequence $\{A_N\}$ is said to be sparsely vanishing if  for every $M>0$ there exists $n_M$ such that, for $N\ge n_M$,
$$\frac{\#\{i\in\{1,\ldots,N\}:\ \sigma_i(A_N)<1/M\}}{N}\le r(M),$$
where $\lim_{M\to\infty}r(M)=0$.
\end{definition}
It is clear that if $\{A_N\}$ is sparsely vanishing then $\{A_N^\dag\}$ is sparsely unbounded; it suffices to recall that the singular values of the pseudo-inverse $A^\dag$ are $1/\sigma_1(A),\ldots,1/\sigma_r(A),0,\ldots,0$, where $\sigma_1(A),\ldots,\sigma_r(A)$ are the nonzero singular values of $A$ ($r={\rm rank}(A)$).

We observe that a matrix sequence having a symbol in the sense of the singular values is always sparsely unbounded and the GLT $*$-algebra is made by sparsely unbounded matrix sequences. Furthermore, if $\left\{A_N\right\}_N\sim_\sigma \chi$ in the sense of
Definition~\ref{def:sv-ev-distribution}, then it is easy to check that
$\left\{A_N\right\}_N$ is sparsely vanishing if and only if $\chi$ vanishes at
most on a set of zero Lebesgue measure.

The following property is also of interest in our context; see \cite{Serra:2003:AOP,Benedusi:2026:AOE}.

\begin{theorem}\label{th:s.u. *-algebra}

   Let ${\mathcal A}=\left\{\{A_N\}_N, A_N\in {\mathcal M}_{N}\right\}$ be the $*$-algebra of matrix sequences and let
\[
{\mathcal SU}=\left\{\{A_N\}_N, A_N\in {\mathcal M}_{N},\ \{A_N\}_N\ s.u. \right\}.
\]
Then ${\mathcal SU}\subset {\mathcal A}$ is a $*$-algebra as well.

In other words, if $\{A_N\},\{A_N'\}$ are sparsely unbounded and $\alpha, \beta$ are given complex constants then
\begin{enumerate}
\item $\{A_N A_N'\}$ is sparsely unbounded,
\item $\{\alpha A_N + \beta A_N'\}$ is sparsely unbounded.
\end{enumerate}
\end{theorem}

\section{Spectral distribution and stability}
\label{sec4}
In this section we analyze the weighted histopolation problem introduced above
from the GLT perspective and study the stability of the associated discrete
operator.

\subsection{Negligible terms and GLT symbols}

We first identify the terms that are negligible in the sense of singular value
distribution, and then compute the symbols of the relevant matrix factors.
This leads to the singular value description of the histopolation matrices and
to stability estimates for the associated discrete operators. We start with the
residual term $R_N$ defined in~\eqref{matRN}.

\begin{proposition}
\label{prop:quant_scaled_RN_zero}
Let $\gamma<1$. Assume that $\alpha,\beta\ge \frac{1}{2}$. Then
$\left\{N^\gamma R_N\right\}_N$ is a zero-distributed sequence, namely
\[
\left\{N^\gamma R_N\right\}_N\sim_{\sigma}0 .
\]
\end{proposition}

\begin{proof}
Fix $\eta$ such that
\[
\max\left\{0,\gamma\right\}<\eta<1,
\]
and set
\begin{equation}\label{newnewmn}
    m_N=\left\lceil N^\eta\right\rceil .
\end{equation}
We decompose
\begin{equation}\label{Ngamm}
    N^\gamma R_N=F_N^{(1)}+F_N^{(2)},
\end{equation}
where $F_N^{(1)}$ contains the first $m_N$ columns of $N^\gamma R_N$ and zeros
elsewhere, while $F_N^{(2)}$ contains the remaining columns. Then
\[
\operatorname{rank}\left(F_N^{(1)}\right)\le m_N.
\]
Let $\varepsilon>0$. We denote by
\[
r=\operatorname{rank}\left(F_N^{(1)}\right),
\quad
s=\#\left\{k=1,\dots,N:\sigma_k\left(F_N^{(2)}\right)>\varepsilon\right\}.
\]
We first show that
\begin{equation}\label{eq:rank_pert_count}
\#\left\{k=1,\dots,N:\sigma_k\left(N^\gamma R_N\right)>\varepsilon\right\}
\le r + s.
\end{equation}
If $r+s\ge N$, then~\eqref{eq:rank_pert_count} is immediate. Hence, we assume
that $r+s<N$. By the Weyl inequality for singular values~\cite{Carl:2009:OWI}, we have
\begin{equation}\label{AN+BN}
\sigma_{r+s+1}\left(F_N^{(1)}+F_N^{(2)}\right)
\le
\sigma_{r+1}\left(F_N^{(1)}\right)
+
\sigma_{s+1}\left(F_N^{(2)}\right).
\end{equation}
Since $\operatorname{rank}\left(F_N^{(1)}\right)=r$, we have
\[
\sigma_{r+1}\left(F_N^{(1)}\right)=0.
\]
Moreover, by the definition of $s$, we have
\[
\sigma_{s+1}\left(F_N^{(2)}\right)\le \varepsilon.
\]
Therefore, by~\eqref{Ngamm} and~\eqref{AN+BN}, we obtain
\[
\sigma_{r+s+1}\left(N^\gamma R_N\right)
=
\sigma_{r+s+1}\left(F_N^{(1)}+F_N^{(2)}\right)
\le \varepsilon.
\]
Hence $N^\gamma R_N$ has at most $r+s$ singular values greater than
$\varepsilon$, and~\eqref{eq:rank_pert_count} follows.

Since
\[
\left\|F_N^{(2)}\right\|_F^2
=
\sum_{k=1}^{N}\sigma_k^2\left(F_N^{(2)}\right),
\]
we have
\[
\left\|F_N^{(2)}\right\|_F^2
\ge
\varepsilon^2
\#\left\{k=1,\dots,N:\sigma_k\left(F_N^{(2)}\right)>\varepsilon\right\}.
\]
Consequently,
\begin{equation}\label{est2}
s=\#\left\{k=1,\dots,N:\sigma_k\left(F_N^{(2)}\right)>\varepsilon\right\}
\le
\frac{\left\|F_N^{(2)}\right\|_F^2}{\varepsilon^2}.
\end{equation}
On the other hand, since $F_N^{(2)}$ contains precisely the columns
$m_N+1,\ldots,N$ of $N^\gamma R_N$, we obtain
\[
\left\|F_N^{(2)}\right\|_F^2
=
N^{2\gamma}
\sum_{j=m_N+1}^{N}
\left\|r_j^{(N)}\right\|_2^2 ,
\]
where the $j$-th column of $R_N$ is
\[
r_j^{(N)}
=
\left[
\widetilde r_j\left(x_0\right),
\widetilde r_j\left(x_1\right),
\dots,
\widetilde r_j\left(x_N\right)
\right]^\top\in\mathbb{R}^{N+1}.
\]
By~\eqref{eq:rjtilde} and Lemma~\ref{lem:weighted_jacobi_uniform_bound}, we have
\[
\sup_{x\in[-1,1]}
\left|\widetilde r_j(x)\right|^2
=
\frac{4}{(j+\alpha+\beta+1)^2}
\sup_{x\in[-1,1]}
\left|P_j^{(\alpha,\beta)}(x)\omega_{\alpha,\beta}(x)\right|^2
\le
\frac{C(\alpha,\beta)}{(j+1)^3}.
\]
It follows that
\begin{equation*}
\left\|r_j^{(N)}\right\|_2^2
=
\sum_{k=0}^{N}
\left|\widetilde r_j\left(x_k\right)\right|^2
\le
C(\alpha,\beta)\frac{N+1}{(j+1)^3}.
\end{equation*}
Therefore,
\begin{equation}\label{new43}
\left\|F_N^{(2)}\right\|_F^2
\le
C(\alpha,\beta)N^{2\gamma}(N+1)
\sum_{j=m_N+1}^{\infty}\frac{1}{(j+1)^3}.
\end{equation}
Using the fact that
\[
(j+1)^2\ge j(j+2),
\]
we have
\[
\frac{1}{(j+1)^3} \le \frac{1}{j(j+1)(j+2)}, \quad j\ge 1.
\]
Then
\begin{eqnarray*}
&&\sum_{j=m_N+1}^{\infty}\frac{1}{(j+1)^3}\le \sum_{j=m_N+1}^{\infty}\frac{1}{j(j+1)(j+2)}\\&=&\frac{1}{2}\sum_{j=m_N+1}^{\infty}\left(\frac{1}{j(j+1)}-\frac{1}{(j+1)(j+2)}\right)=\frac{1}{2(m_N+1)(m_N+2)}.
\end{eqnarray*}
Therefore, by~\eqref{new43} and~\eqref{newnewmn}, we have
\[
\left\|F_N^{(2)}\right\|_F^2
\le
C(\alpha,\beta)
\frac{N^{2\gamma}(N+1)}
{\left(m_N+1\right)\left(m_N+2\right)}
=
 O\left(N^{2\gamma+1-2\eta}\right).
\]
Combining this estimate with~\eqref{eq:rank_pert_count} and~\eqref{est2}, we get
\begin{eqnarray*}
&& \#\left\{k=1,\dots,N:\sigma_k\left(N^\gamma R_N\right)>\varepsilon\right\}\le\\
&& \operatorname{rank}\left(F_N^{(1)}\right)+\#\left\{k=1,\dots,N:\sigma_k\left(F_N^{(2)}\right)>\varepsilon\right\}\le\\
&&m_N
+
\frac{\left\|F_N^{(2)}\right\|_F^2}{\varepsilon^2}
=
 O\left(N^\eta\right)
+
 O\left(N^{2\gamma+1-2\eta}\right).
\end{eqnarray*}
Since
\[
\eta<1,
\quad
2\gamma+1-2\eta<1,
\]
there exists $\rho\in(0,1)$ such that
\[
\#\left\{k=1,\dots,N:\sigma_k\left(N^\gamma R_N\right)>\varepsilon\right\}
=
 O\left(N^\rho\right).
\]
In particular,
\[
\lim_{N\to\infty}\frac{\#\left\{k=1,\dots,N:\sigma_k\left(N^\gamma R_N\right)>\varepsilon\right\}}{N}=0,
\]
for any $\varepsilon>0$. By Theorem~\ref{thmGaroni}, we conclude that
\[
\left\{N^\gamma R_N\right\}_N\sim_{\sigma}0 .
\]
\end{proof}

\begin{corollary}
Let $\gamma<1$. Assume that $\alpha,\beta\ge \frac{1}{2}$. Then
\[
\left\{N^\gamma\left(\Psi_N-I_N^{\mathrm{ext}}T_N^{(J)}\right)\right\}_N
\sim_{\sigma}0 .
\]
Consequently, the two sequences
\[
\left\{N^\gamma\Psi_N\right\}_N
\quad\text{and}\quad
\left\{N^\gamma I_N^{\mathrm{ext}}T_N^{(J)}\right\}_N
\]
differ by a zero-distributed sequence. Hence, whenever one of them admits an
asymptotic singular value distribution, the other one admits the same
distribution.
\end{corollary}

\begin{proof}
By Theorem~\ref{thm:struct-tridiag}, we have
\[
\Psi_N-I_N^{\mathrm{ext}}T_N^{(J)}=R_N .
\]
The result follows immediately from Proposition~\ref{prop:quant_scaled_RN_zero}.
\end{proof}

\begin{corollary}\label{prop:RN_zero}
Assume that $\alpha,\beta\ge \frac{1}{2}$. Then $\left\{R_N\right\}_N$ is a
zero-distributed sequence, namely
\[
\left\{R_N\right\}_N\sim_{\sigma}0.
\]
\end{corollary}

\begin{proof}
The conclusion follows from Proposition~\ref{prop:quant_scaled_RN_zero} by
choosing $\gamma=0$.
\end{proof}

We now compute the asymptotic singular value symbol of the matrix sequence
\[
\left\{N T_N^{(J)}\right\}_N,
\]
where
$T_N^{(J)}\in\mathbb{R}^{(N+2)\times N}$ is the matrix defined
in~\eqref{matrixTNj}.

\begin{theorem}
\label{thm:toeplitz-symbol-TJ}
Let $\alpha,\beta>0$ and set
\[
\sigma:=\alpha+\beta,\quad \delta:=\alpha-\beta.
\]
Then
\[
\left\{N T_N^{(J)}\right\}_N\sim_{\mathrm{GLT}} \kappa(y,\theta),
\]
where
\begin{equation}\label{funkappa}
\kappa(y,\theta)
=
\frac{1}{y}
\left[
\sigma+2\delta e^{i\theta}+\sigma e^{2i\theta}
\right],
\quad
(y,\theta)\in (0,1]\times[-\pi,\pi].
\end{equation}
\end{theorem}

\begin{proof}
From the definitions~\eqref{aj}-\eqref{cj}, we have
\[
a_j=\frac{1}{2}+O\left(\frac{1}{j}\right),
\quad
b_j=O\left(\frac{1}{j^2}\right),
\quad
c_j=\frac{1}{2}+O\left(\frac{1}{j}\right).
\]
Consequently, by~\eqref{ujdjlj}, we get
\begin{equation*}
j u_j=\sigma+O\left(\frac{1}{j}\right),
\quad
j d_j=2\delta+O\left(\frac{1}{j}\right),
\quad
j\ell_j=\sigma+O\left(\frac{1}{j}\right).
\end{equation*}
Equivalently,
\begin{equation}
    \label{juj}
    u_j=\frac{\sigma}{j}+O\left(\frac{1}{j^2}\right),
\quad
d_j=\frac{2\delta}{j}+O\left(\frac{1}{j^2}\right),
\quad
\ell_j=\frac{\sigma}{j}+O\left(\frac{1}{j^2}\right).
\end{equation}
For $r=0,1,2$, let $E_{N,r}\in\mathbb{R}^{(N+2)\times N}$ be defined by
\[
\left[E_{N,r}\right]_{j+r,j}=1,\quad j=1,\dots,N,
\]
with all the remaining entries equal to zero. We also set
\[
D_N=\operatorname{diag}\left(\frac{N}{1},\frac{N}{2},\dots,\frac{N}{N}\right).
\]
Define
\[
B_N
=
\left(
\sigma E_{N,0}+2\delta E_{N,1}+\sigma E_{N,2}
\right)D_N .
\]
By construction, the $j$-th column of $B_N$ has exactly three nonzero entries,
namely
\[
\left[B_N\right]_{j,j}=\frac{\sigma N}{j},\quad
\left[B_N\right]_{j+1,j}=\frac{2\delta N}{j},\quad
\left[B_N\right]_{j+2,j}=\frac{\sigma N}{j}.
\]
On the other hand, by the definition of $T_N^{(J)}$, we have
\[
\left[NT_N^{(J)}\right]_{j,j}=N\ell_j,\quad
\left[NT_N^{(J)}\right]_{j+1,j}=Nd_j,\quad
\left[NT_N^{(J)}\right]_{j+2,j}=Nu_j.
\]
Thus the $j$-th column of $N T_N^{(J)}-B_N$ may be nonzero only in the rows
$j,j+1,j+2$, where its entries are respectively
\[
N\ell_j-\frac{\sigma N}{j},\quad
Nd_j-\frac{2\delta N}{j},\quad
Nu_j-\frac{\sigma N}{j}.
\]
Using~\eqref{juj}, there exists a constant $C>0$, independent of $N$ and $j$,
such that
\begin{equation}\label{dscaaaa}
\left|
N\ell_j-\frac{\sigma N}{j}
\right|
+
\left|
Nd_j-\frac{2\delta N}{j}
\right|
+
\left|
Nu_j-\frac{\sigma N}{j}
\right|
\le
C\frac{N}{j^2}.
\end{equation}
We now show that
\[
\left\{N T_N^{(J)}-B_N\right\}_N\sim_{\sigma}0.
\]
Let
\[
m_N=\left\lceil N^{2/3}\right\rceil.
\]
We decompose
\[
N T_N^{(J)}-B_N=F_N^{(1)}+F_N^{(2)},
\]
where $F_N^{(1)}$ contains the first $m_N$ columns and $F_N^{(2)}$ the remaining
ones. Hence
\[
\operatorname{rank}\left(F_N^{(1)}\right)\le m_N,
\]
and therefore
\[
\lim_{N\to\infty}
\frac{\operatorname{rank}\left(F_N^{(1)}\right)}{N}=0.
\]
On the other hand, by~\eqref{dscaaaa}, every nonzero entry of $F_N^{(2)}$ is
bounded in modulus by
\[
C\frac{N}{m_N^2}.
\]
Since $F_N^{(2)}$ has at most three nonzero entries in each column and at most
three nonzero entries in each row, we have
\[
\left\|F_N^{(2)}\right\|
\le
\sqrt{
\left\|F_N^{(2)}\right\|_1
\left\|F_N^{(2)}\right\|_{\infty}
}
\le
3C\frac{N}{m_N^2}.
\]
Since $m_N=\left\lceil N^{2/3}\right\rceil$, it follows that
\[
\lim_{N\to\infty}\left\|F_N^{(2)}\right\|=0.
\]
By Theorem~\ref{thmGaroni}, we conclude that
\begin{equation}\label{eq:NTJ_minus_BN_zero}
\left\{N T_N^{(J)}-B_N\right\}_N\sim_{\sigma}0.
\end{equation}
Hence, we have
\[
\{N T_N^{(J)}-B_N\}_N\sim_{\mathrm{GLT}}0.
\]
It remains to determine the symbol of $B_N$. The matrix
\[
\sigma E_{N,0}+2\delta E_{N,1}+\sigma E_{N,2}
\]
has Toeplitz symbol
\[
p(\theta)=\sigma+2\delta e^{i\theta}+\sigma e^{2i\theta}.
\]
Moreover, the diagonal matrix $D_N$ is the sampling matrix of the function
$y\mapsto y^{-1}$ on $(0,1]$. Therefore, by the standard GLT algebra for
products of Toeplitz sequences and diagonal sampling sequences, we have
\[
\left\{B_N\right\}_N
\sim_{\mathrm{GLT}}
\frac{p(\theta)}{y}.
\]
Then, by~\eqref{eq:NTJ_minus_BN_zero}, we obtain
\[
\left\{N T_N^{(J)}\right\}_N
\sim_{\mathrm{GLT}}
\frac{1}{y}
\left[
\sigma+2\delta e^{i\theta}+\sigma e^{2i\theta}
\right].
\]
Finally, since
\[
\sigma+2\delta e^{i\theta}+\sigma e^{2i\theta}
=
2e^{i\theta}\left(\delta+\sigma\cos\theta\right),
\]
from~\eqref{funkappa}, we have
\[
|\kappa(y,\theta)|
=
\frac{2}{y}
\left|
\delta+\sigma\cos\theta
\right|.
\]
Since $\sigma=\alpha+\beta>0$ and $|\delta|<\sigma$, the equation
\[
\delta+\sigma\cos\theta=0
\]
has at most finitely many solutions in $[-\pi,\pi]$. Hence $\kappa$ is nonzero
a.e. on $(0,1]\times[-\pi,\pi]$.
\end{proof}

We now prove a few auxiliary lemmas that are useful for the spectral analysis 
of the weighted histopolation matrix $H_N$. The first one concerns the weighted
primitives $I_j$ defined in~\eqref{eq:Ij}.

\begin{lemma}\label{lem:primitive_I_uniform_bound}
Assume that $\alpha,\beta>\frac{1}{2}$. Then there exists a constant
$C(\alpha,\beta)>0$ such that
\begin{equation*}
\sup_{x\in[-1,1]}\left|I_j(x)\right|
\le
\frac{C(\alpha,\beta)}{\sqrt{j+1}},
\quad j\ge0.
\end{equation*}
\end{lemma}

\begin{proof}
By definition, for any $x\in[-1,1]$ we have
\[
\left|I_j(x)\right|
\le
\int_{-1}^1
\left|P_j^{(\alpha,\beta)}(t)\right|
\omega_{\alpha-1,\beta-1}(t)dt.
\]
By the uniform estimate for orthonormal Jacobi polynomials, there exists a
constant $C_0(\alpha,\beta)>0$ such that
\[
(1-t)^{\alpha+\frac{1}{2}}(1+t)^{\beta+\frac{1}{2}}
\left|\widehat P_j^{(\alpha,\beta)}(t)\right|^2
\le
C_0(\alpha,\beta),
\quad t\in[-1,1],\quad j\ge0.
\]
Moreover,
\[
P_j^{(\alpha,\beta)}(t)=K_j^{1/2}\widehat P_j^{(\alpha,\beta)}(t),
\]
where $K_j$ is defined in~\eqref{Kj}. Hence
\[
\left|P_j^{(\alpha,\beta)}(t)\right|
\le
\sqrt{C_0(\alpha,\beta)}
K_j^{1/2}
(1-t)^{-\frac{\alpha}{2}-\frac14}
(1+t)^{-\frac{\beta}{2}-\frac14}.
\]
Multiplying by $\omega_{\alpha-1,\beta-1}(t)$, we obtain
\[
\left|P_j^{(\alpha,\beta)}(t)\right|
\omega_{\alpha-1,\beta-1}(t)
\le
\sqrt{C_0(\alpha,\beta)}
K_j^{1/2}
(1-t)^{\frac{\alpha}{2}-\frac54}
(1+t)^{\frac{\beta}{2}-\frac54}.
\]
Since $\alpha,\beta>\frac{1}{2}$, we have
\[
\frac{\alpha}{2}-\frac54>-1,
\quad
\frac{\beta}{2}-\frac54>-1.
\]
Therefore
\[
(1-t)^{\frac{\alpha}{2}-\frac54}
(1+t)^{\frac{\beta}{2}-\frac54}
\in L^1(-1,1).
\]
Hence there exists a constant $C_1(\alpha,\beta)>0$ such that
\[
\sup_{x\in[-1,1]}\left|I_j(x)\right|
\le
C_1(\alpha,\beta)K_j^{1/2}.
\]
Using~\eqref{eq:Kj_bound} and renaming the constant, we finally get
\[
\sup_{x\in[-1,1]}\left|I_j(x)\right|
\le
\frac{C(\alpha,\beta)}{\sqrt{j+1}},
\quad j\ge0.
\]
The proof is complete.
\end{proof}

The uniform bound in Lemma~\ref{lem:primitive_I_uniform_bound} allows us to
prove that the matrix $I_N^{\mathrm{ext}}$, defined in~\eqref{INext}, becomes negligible after scaling by
$N$.
\begin{proposition}\label{prop:INext_over_N_zero}
Assume that $\alpha,\beta>\frac{1}{2}$. Then
$\left\{\frac{I_N^{\mathrm{ext}}}{N}\right\}_N$ is a zero-distributed sequence, namely
\[
\left\{\frac{I_N^{\mathrm{ext}}}{N}\right\}_N\sim_{\sigma}0.
\]
\end{proposition}

\begin{proof}
By the definition~\eqref{INext}, we have
\[
\left[I_N^{\mathrm{ext}}\right]_{k+1,r+1}
=
I_r\left(x_k\right),
\]
for $k=0,\dots,N,$ $r=0,\dots,N+1.$
Using Lemma~\ref{lem:primitive_I_uniform_bound}, we get
\[
\left|I_r\left(x_k\right)\right|
\le \sup_{x\in[-1,1]}\left|I_r\left(x\right)\right|\le \frac{C(\alpha,\beta)}{\sqrt{r+1}},
\]
for $k=0,\dots,N,$ $ r=0,\dots,N+1.$
Therefore
\begin{eqnarray*}
\left\|I_N^{\mathrm{ext}}\right\|_F^2
&=&
\sum_{k=0}^{N}\sum_{r=0}^{N+1}
\left|I_r\left(x_k\right)\right|^2\le
\left(C(\alpha,\beta)\right)^2
\sum_{k=0}^{N}\sum_{r=0}^{N+1}
\frac{1}{r+1}\\
&=&
\left(C(\alpha,\beta)\right)^2
(N+1)
\sum_{r=0}^{N+1}
\frac{1}{r+1}.
\end{eqnarray*}
Since
\[
\sum_{r=0}^{N+1}\frac{1}{r+1}
=
\sum_{r=1}^{N+2}\frac{1}{r}
\le
1+\log(N+2),
\]
we obtain
\[
\left\|I_N^{\mathrm{ext}}\right\|_F^2
\le
\left(C(\alpha,\beta)\right)^2
(N+1)\left(1+\log(N+2)\right).
\]
Consequently, we have
\[
\left\|\frac{I_N^{\mathrm{ext}}}{N}\right\|_F^2
\le
\left(C(\alpha,\beta)\right)^2
\frac{N+1}{N^2}
\left(1+\log(N+2)\right).
\]
Hence
\[
\lim_{N\to\infty}
\left\|\frac{I_N^{\mathrm{ext}}}{N}\right\|_F
=0.
\]
Since the spectral norm is bounded by the Frobenius norm, it follows that
\[
\lim_{N\to\infty}
\left\|\frac{I_N^{\mathrm{ext}}}{N}\right\|
=0.
\]
Thus, by Theorem~\ref{thmGaroni}, we conclude that
\[
\left\{\frac{I_N^{\mathrm{ext}}}{N}\right\}_N\sim_{\sigma}0.
\]
\end{proof}

We now compute the symbol of the matrix $\Delta_N$, defined
in~\eqref{mat:DeltaN}. Its symbol is determined by the local size of the mesh.
In the following, we assume that the nodes are obtained from a uniform
reference grid through a smooth increasing map. More precisely, let
\[
z_i=-1+2u_i,\ u_i=\frac{i}{N},\quad i=0,\ldots,N,
\]
and let $\hat g\in C^1([-1,1])$ be such that
\[
\hat g(-1)=-1,\quad \hat g(1)=1,\quad \hat g'(y)>0\ \ a.e.,\quad y\in[-1,1].
\]
For every $N\ge1$, we define
\[
x_i=\hat g\left(z_i\right)=g(u_i),\quad i=0,\ldots,N,
\]
and set
\[
h_i=x_i-x_{i-1}
=
g\left(u_i\right)
-
g\left(u_{i-1}\right),
\quad i=1,\ldots,N.
\]

\begin{lemma}\label{lem:DeltaN_symbol}
Under the assumptions stated above on the mesh, we have
\[
\left\{\frac{\Delta_N}{N}\right\}_N
\sim_{\mathrm{GLT}}
\frac{1-e^{i\theta}}{g'(y)},
\quad (y,\theta)\in[0,1]\times[-\pi,\pi].
\]
\end{lemma}

\begin{proof}
By definition of $\Delta_N$ in~\eqref{mat:DeltaN}, taking into account the
definition of $T_N(f)$ in~\eqref{toe:def}-\eqref{fou:def}, with specific
reference to the polynomial setting in~\eqref{toe:def-pol}-\eqref{fou:def-pol},
we have
\[
\frac{\Delta_N}{N}
=
\operatorname{diag}((Nh_i)^{-1})B_N,
\]
where $B_N\in\mathbb{R}^{N\times(N+1)}$ is the rectangular backward-difference
matrix
\[
B_N=\left[-e_1\ \ T_N(1-e^{i\theta})\right],
\]
with $e_1$ denoting the first vector of the canonical basis of $\mathbb{R}^N$.
Indeed, the $i$-th row of $B_N$ contains the two nonzero entries $-1$ and $1$
in the positions corresponding to $u_{i-1}$ and $u_i$, respectively.
The additional boundary column $-e_1$ has rank one and therefore does not
affect the GLT symbol. Hence
\[
\{B_N\}_N\sim_{\mathrm{GLT}}1-e^{i\theta}.
\]

The diagonal component can be written as
\[
\operatorname{diag}((Nh_i)^{-1})=[\operatorname{diag}(Nh_i)]^{-1},
\]
where
\[
\operatorname{diag}(Nh_i)
=
\operatorname{diag}\left(N\left(g(u_i)-g(u_{i-1})\right)\right).
\]
By a Taylor expansion, we can write
\[
N\left(g(u_i)-g(u_{i-1})\right)=g'(u_i)+\epsilon(i,N),
\]
where
\[
\lim_{N\to\infty}\max_i|\epsilon(i,N)|=0.
\]
Thus
\[
\operatorname{diag}(Nh_i)=D_N(g')+Z_N,
\]
where $Z_N$ is a diagonal matrix with infinitesimal entries. As a consequence,
$\{Z_N\}_N\sim_{\sigma}0$ by direct inspection and
$\{Z_N\}_N\sim_{\mathrm{GLT}}0$ by axiom \textbf{GLT 2}. Since, again by axiom
\textbf{GLT 2},
\[
\{D_N(g')\}_N\sim_{\mathrm{GLT}}g'(y),
\]
and since $g'>0$ a.e. by the assumptions on the mesh, we get
\[
\left\{[\operatorname{diag}(Nh_i)]^{-1}\right\}_N
\sim_{\mathrm{GLT}}
\frac{1}{g'(y)}.
\]
The conclusion follows by applying the algebraic rules in axiom \textbf{GLT 3} to
\[
\frac{\Delta_N}{N}
=
[\operatorname{diag}(Nh_i)]^{-1}B_N.
\]
\end{proof}

\begin{theorem}\label{thmStef}
Assume that $\alpha,\beta>\frac{1}{2}$ and that the mesh satisfies the
assumptions of Lemma~\ref{lem:DeltaN_symbol}. Then the weighted Jacobi
histopolation matrix sequence $\left\{\frac{H_N}{N}\right\}_N$ is zero-distributed,
namely
\[
\left\{\frac{H_N}{N}\right\}_N \sim_{\mathrm{GLT}}
0.
\]
\end{theorem}

\begin{proof}
By Theorem~\ref{thm:exact-factorization}, we have
\[
H_N=\Delta_N\Psi_N.
\]
Moreover, by Theorem~\ref{thm:struct-tridiag}, the primitive sampling matrix
admits the decomposition
\[
\Psi_N=R_N+I_N^{\mathrm{ext}}T_N^{(J)}.
\]
Hence
\[
H_N
=
\Delta_NR_N
+
\Delta_NI_N^{\mathrm{ext}}T_N^{(J)}.
\]

We first consider the term $\Delta_NR_N$. By Lemma~\ref{lem:DeltaN_symbol},
\[
\left\{\frac{\Delta_N}{N}\right\}_N
\sim_{\mathrm{GLT}}
\frac{1-e^{i\theta}}{g'(y)}.
\]
On the other hand, by Corollary~\ref{prop:RN_zero}, we have
\[
\left\{R_N\right\}_N\sim_\sigma0
\]
i.e. $\left\{R_N\right\}_N\sim_{\mathrm{GLT}}0$ by axiom \textbf{GLT 2}.

Using the GLT product rule, that is part 3 of axiom \textbf{GLT 3}, we obtain
\[
\left\{\frac{\Delta_N}{N}R_N\right\}_N
\sim_{\mathrm{GLT}}
0,
\]
so that, by \textbf{GLT 1}, we deduce
\[
\left\{\frac{\Delta_N}{N}R_N\right\}_N\sim_\sigma0.
\]

We now consider the second term. We write
\[
\frac{\Delta_N}{N}I_N^{\mathrm{ext}}T_N^{(J)}
=
\frac{\Delta_N}{N}
\left(\frac{1}{N}I_N^{\mathrm{ext}}\right)
\left(NT_N^{(J)}\right).
\]
By Proposition~\ref{prop:INext_over_N_zero},
\[
\left\{\frac{1}{N}I_N^{\mathrm{ext}}\right\}_N
\sim_{\mathrm{GLT}}0.
\]
Furthermore, by Theorem~\ref{thm:toeplitz-symbol-TJ},
\[
\left\{N T_N^{(J)}\right\}_N\sim_{\mathrm{GLT}} \frac{1}{y}
\left[
\sigma+2\delta e^{i\theta}+\sigma e^{2i\theta}
\right],
\quad
(y,\theta)\in (0,1]\times[-\pi,\pi].
\]
Therefore, another application of the GLT product rule (part 3 of \textbf{GLT 3}) leads to
\[
\left\{
\frac{\Delta_N}{N}
\left(\frac{1}{N}I_N^{\mathrm{ext}}\right)
\left(NT_N^{(J)}\right)
\right\}_N
\sim_{\mathrm{GLT}}
0.
\]
Consequently, by \textbf{GLT 1}, we find
\[
\left\{\frac{\Delta_N}{N} I_N^{\mathrm{ext}}T_N^{(J)}\right\}_N
\sim_\sigma0.
\]
Combining the two estimates and using \textbf{GLT 2}, we infer
\[
\left\{\frac{H_N}{N}\right\}_N
=
\left\{
\frac{\Delta_N}{N}R_N
+
\frac{\Delta_N}{N}I_N^{\mathrm{ext}}T_N^{(J)}
\right\}_N
\sim_{\mathrm{GLT}}0.
\]
Thus, by the singular value distribution property of GLT sequences, i.e., axiom \textbf{GLT 1}, we finally deduce
\[
\left\{\frac{H_N}{N}\right\}_N\sim_\sigma0.
\]
\end{proof}

\begin{remark}\label{two-sided ideal}
In the previous theorem we repeatedly used the fact that zero-distributed
matrix sequences form a two-sided ideal in the GLT $*$-algebra. In other words, if $\{Z_N\}\sim_{\mathrm{GLT}}0$ and $\{A_N\}$ is any GLT matrix sequence, then
$\{A_NZ_N\}\sim_{\mathrm{GLT}}0$ and
$\{Z_NA_N\}\sim_{\mathrm{GLT}}0$ by axiom \textbf{GLT 3};
that is, $\{A_NZ_N\}$ and $\{Z_NA_N\}$ are both zero-distributed. Such a strong property holds for the large $*$-algebra of sparsely unbounded matrix sequences (which possess less structure since there is no  $*$-algebra homomorphism with the space of  symbols): indeed, as recalled in Theorem \ref{th:s.u. *-algebra}, the zero-distributed matrix sequences are a two-sided ideal for the sparsely unbounded matrix sequences. This remark is of interest in our setting, since looking at Lemma \ref{lem:DeltaN_symbol} it allows us to impose weaker assumptions on the values $h_i$. More precisely if $\left\{\frac{\Delta_N}{N}\right\}_N$ is sparsely unbounded then we get again the result of Theorem \ref{thmStef} that is
\[
\left\{\frac{H_N}{N}\right\}_N \sim_{\mathrm{GLT}}
0,
\]
by exploiting the factorization $H_N=\Delta_N\Psi_N$ and the fact that $\{Z_N\}\sim_{\mathrm{GLT}} 0$.

Finally, if $g'$ is continuous and strictly positive on $[0,1]$, then the grid is
quasi-uniform. In this case, since $g'$ is bounded away from zero, the inverse
sequence $\{D_N^{-1}(g')\}_N$ is not only sparsely unbounded, but also uniformly
bounded in spectral norm.
\end{remark}

We introduce the matrix
\begin{equation}\label{eq:matDh}
D_h=\operatorname{diag}\left(h_1,\ldots,h_N\right),
\quad
h_i=x_i-x_{i-1}.
\end{equation}

\begin{lemma}\label{lem:H_weighted_column_bound}
Assume that $\alpha,\beta>\frac{1}{2}$. Then there exists a constant $C(\alpha,\beta)>0$, such that
\begin{equation}\label{eq:weighted_H_column_bound}
\left\|\left(D_h^{1/2}H_N\right)e_j\right\|_2^2
\le
\frac{C(\alpha,\beta)}{j},
\quad j=1,\ldots,N ,
\end{equation}
where $\{e_j\}_{j=1}^N$ denotes the canonical basis of $\mathbb{R}^N$.
\end{lemma}

\begin{proof}
By the definition of $H_N$ with respect to the basis~\eqref{jacbasiabov}, we have
\[
[H_N]_{i,j}
=
\frac1{h_i}\int_{s_i} P_{j-1}^{(\alpha+1,\beta+1)}(t)\omega_{\alpha,\beta}(t) dt,
\quad i=1,\ldots,N.
\]
Hence, using the Cauchy--Schwarz inequality, we get
\begin{eqnarray*}
h_i \left|[H_N]_{i,j}\right|^2
&=&
h_i
\left|
\frac1{h_i}\int_{s_i}P_{j-1}^{(\alpha+1,\beta+1)}(t)\omega_{\alpha,\beta}(t)dt
\right|^2\\
&\le&
\int_{s_i}\left|P_{j-1}^{(\alpha+1,\beta+1)}(t)\omega_{\alpha,\beta}(t)\right|^2dt
\end{eqnarray*}
Summing over $i=1,\ldots,N,$ we get
\begin{eqnarray}\notag
\left\|\left(D_h^{1/2}H_N\right)e_j\right\|_2^2
&=&
\sum_{i=1}^N h_i \left|[H_N]_{i,j}\right|^2\\
&\le&
\int_{-1}^1
\left|P_{j-1}^{(\alpha+1,\beta+1)}(t)\right|^2
\left(\omega_{\alpha,\beta}(t)\right)^2dt.\label{eq:H_column_L2_reduction}
\end{eqnarray}
By the uniform estimate for orthonormal Jacobi polynomials, applied with
parameters $\alpha+1$ and $\beta+1$, there exists a constant $C_0(\alpha,\beta)>0$ such that
\[
(1-t)^{\alpha+\frac{3}{2}}(1+t)^{\beta+\frac{3}{2}}
\left|\widehat P_{j-1}^{(\alpha+1,\beta+1)}(t)\right|^2
\le
C_0(\alpha,\beta),
\quad t\in(-1,1),\quad j\ge1 .
\]
Therefore, denoting by $K_{j-1}^{(\alpha+1,\beta+1)}$ the corresponding Jacobi
norm, we have
\begin{eqnarray*}
\left|P_{j-1}^{(\alpha+1,\beta+1)}(t)\right|^2
\left(\omega_{\alpha,\beta}(t)\right)^2&=&
K_{j-1}^{(\alpha+1,\beta+1)}
\left|\widehat P_{j-1}^{(\alpha+1,\beta+1)}(t)\right|^2
(1-t)^{2\alpha}(1+t)^{2\beta}
\\
&\le&
C_0(\alpha,\beta)K_{j-1}^{(\alpha+1,\beta+1)}
(1-t)^{\alpha-\frac{3}{2}}
(1+t)^{\beta-\frac{3}{2}}.
\end{eqnarray*}
Since $\alpha,\beta>\frac{1}{2}$, the function
\[
(1-t)^{\alpha-\frac{3}{2}}(1+t)^{\beta-\frac{3}{2}}
\]
belongs to $L^1(-1,1)$. Hence
\begin{equation}\label{aus1}
    \int_{-1}^1
\left|P_{j-1}^{(\alpha+1,\beta+1)}(t)\right|^2
\left(\omega_{\alpha,\beta}(t)\right)^2dt
\le
C_1(\alpha,\beta)K_{j-1}^{(\alpha+1,\beta+1)} .
\end{equation}
Using~\eqref{eq:Kj_bound}, we obtain
\[
K_{j-1}^{(\alpha+1,\beta+1)}
\le
\frac{C_2(\alpha,\beta)}{j}.
\]
Combining this estimate with~\eqref{eq:H_column_L2_reduction} and~\eqref{aus1}, we have
\[
\left\|\left(D_h^{1/2}H_N\right)e_j\right\|_2^2
\le
\frac{C(\alpha,\beta)}{j},
\quad j=1,\ldots,N.
\]
This proves~\eqref{eq:weighted_H_column_bound}.
\end{proof}

\begin{theorem}\label{thm:scaled_weighted_H_zero}
Assume that $\alpha,\beta>\frac{1}{2}$. Then, for any fixed
$\gamma<\frac{1}{2}$, the sequence
\[
\left\{N^\gamma D_h^{1/2}H_N\right\}_N\sim_{\mathrm{GLT}}
0.
\]
\end{theorem}

\begin{proof}
Let $\gamma<\frac{1}{2}$. Fix $\eta\in(0,1)$ and set
\begin{equation}\label{mNnews}
    m_N=\left\lceil N^\eta\right\rceil .
\end{equation}
We decompose
\[
N^\gamma D_h^{1/2}H_N=F_N^{(1)}+F_N^{(2)},
\]
where $F_N^{(1)}$ contains the first $m_N$ columns of
$N^\gamma D_h^{1/2}H_N$ and zeros elsewhere, while $F_N^{(2)}$ contains the remaining
columns. Then
\begin{equation}\label{rrankF1}
\operatorname{rank}\left(F_N^{(1)}\right)\le m_N,
\end{equation}
and therefore
\[
\lim_{N\to\infty}\frac{\operatorname{rank}\left(F_N^{(1)}\right)}{N}=0.
\]
Let $\varepsilon>0$ and let
\[
r=\operatorname{rank}\left(F_N^{(1)}\right), \quad s=\#\left\{k=1,\dots,N:\sigma_k\left(F_N^{(2)}\right)>\varepsilon\right\}.
\]
Using Weyl's inequality, in analogy with  Proposition~\ref{prop:quant_scaled_RN_zero}, we obtain
\begin{equation}\label{eq:weighted_H_count_split}
\#\left\{k=1,\dots,N:\sigma_k\left(N^\gamma D_h^{1/2}H_N\right)>\varepsilon\right\}
\le
r+s.
\end{equation}
Moreover
\begin{equation}\label{sqe}
    s=\#\left\{k=1,\dots,N:\sigma_k\left(F_N^{(2)}\right)>\varepsilon\right\}
\le
\frac{\left\|F_N^{(2)}\right\|_F^2}{\varepsilon^2}.
\end{equation}
By Lemma~\ref{lem:H_weighted_column_bound}, we have
\begin{eqnarray}\notag
\left\|F_N^{(2)}\right\|_F^2
&=&
N^{2\gamma}
\sum_{j=m_N+1}^{N}
\left\|\left(D_h^{1/2}H_N\right)e_j\right\|_2^2
\\
&\le&
C(\alpha,\beta)N^{2\gamma}
\sum_{j=m_N+1}^{N}\frac{1}{j}
\le
C(\alpha,\beta)N^{2\gamma}\log N . \label{sacccv}
\end{eqnarray}
Using~\eqref{eq:weighted_H_count_split} together with~\eqref{rrankF1}, \eqref{sqe}, and~\eqref{sacccv}, we obtain
\[
\#\left\{k=1,\dots,N:\sigma_k\left(N^\gamma D_h^{1/2}H_N\right)>\varepsilon\right\}
\le
m_N
+
\frac{C(\alpha,\beta)}{\varepsilon^2}
N^{2\gamma}\log N .
\]
Dividing by $N$ and using~\eqref{mNnews}, we have
\[
\frac{\#\left\{k=1,\dots,N:\sigma_k\left(N^\gamma D_h^{1/2}H_N\right)>\varepsilon\right\}}{N}
\le
N^{\eta-1}
+
\frac{C(\alpha,\beta)}{\varepsilon^2}
N^{2\gamma-1}\log N .
\]
Since $\eta<1$ and $\gamma<\frac{1}{2}$, the right-hand side tends to zero.
Therefore,
\[
\lim_{N\to\infty}\frac{\#\left\{k=1,\dots,N:\sigma_k\left(N^\gamma D_h^{1/2}H_N\right)>\varepsilon\right\}}{N}
=0,
\]
for any $\varepsilon>0$. By Theorem~\ref{thmGaroni}, we conclude that
\[
\left\{N^\gamma D_h^{1/2}H_N\right\}_{N}\sim_{\mathrm{GLT}}
0.
\]
\end{proof}

The next result shows that, if the mesh is assumed in addition to be quasi-uniform (see also Remark \ref{two-sided ideal}), the conclusion of Theorem~\ref{thmStef} can be further improved.

\begin{corollary}\label{cor:scaled_H_zero_quasi_uniform}
Assume that $\alpha,\beta>\frac{1}{2}$ and that the mesh is quasi-uniform, namely
there exist constants $c_h,C_h>0$, independent of $N$, such that
\begin{equation}\label{eq:quasi_uniform_mesh}
\frac{c_h}{N}\le h_i\le \frac{C_h}{N},
\quad i=1,\ldots,N .
\end{equation}
Then, for any $\gamma<0$, the sequence
\[
\left\{N^\gamma H_N\right\}_N\sim_{\mathrm{GLT}}
0.
\]
\end{corollary}

\begin{proof}
By~\eqref{eq:matDh} and~\eqref{eq:quasi_uniform_mesh}, the diagonal matrices
\[
S_N=N^{-1/2}D_h^{-1/2}
\]
are uniformly bounded, since
\begin{equation}\label{bounded}
    \left\|S_N\right\|_2
=
N^{-1/2}\max_{1\le i\le N}h_i^{-1/2}
\le
\frac1{\sqrt{c_h}} .
\end{equation}
Moreover
\[
N^\gamma H_N
=
S_N N^{\gamma+\frac{1}{2}}D_h^{1/2}H_N.
\]
Since $\gamma<0$, we have
\[
\gamma+\frac{1}{2}<\frac{1}{2} .
\]
Hence, by Theorem~\ref{thm:scaled_weighted_H_zero}, we get
\[
\left\{N^{\gamma+\frac{1}{2}}D_h^{1/2}H_N\right\}_N\sim_{\mathrm{GLT}}0.
\]
Since by~\eqref{bounded} $\left\{S_N\right\}_N$ is uniformly bounded in spectral norm, we conclude
\[
\left\{N^\gamma H_N\right\}_N\sim_{\mathrm{GLT}}
0.
\]
\end{proof}

\begin{theorem}\label{thm:H_log_gamma_zero}
Assume that $\alpha,\beta>\frac{1}{2}$ and that the mesh is quasi-uniform. Then, for any fixed $\gamma>0$, we have
\[
\left\{\frac{H_N}{(\log N)^\gamma}\right\}_N\sim_{\mathrm{GLT}}
0.
\]
\end{theorem}

\begin{proof}
Since $D_h^{-1/2}$ is diagonal, the quasi-uniformity assumption gives
\[
\left\|D_h^{-1/2}\right\|_2^2
=
\max_{1\le i\le N} h_i^{-1}
\le
\frac{N}{c_h}.
\]
Hence, by Lemma~\ref{lem:H_weighted_column_bound}, for any $j=1,\ldots,N$, we obtain
\begin{equation}\label{eq:H_column_euclidean_bound}
\left\|H_Ne_j\right\|_2^2
\le
\left\|D_h^{-1/2}\right\|_2^2
\left\|\left(D_h^{1/2}H_N\right)e_j\right\|_2^2
\le
C(\alpha,\beta,c_h)\frac{N}{j}.
\end{equation}
Fix $\gamma>0$ and let $\eta>0$. We set
\begin{equation}\label{eq:mN_log_cutoff}
m_N=
\left\lfloor
\frac{N}{(\log N)^\eta}
\right\rfloor.
\end{equation}
We decompose
\[
\frac{H_N}{(\log N)^\gamma}
=
F_N^{(1)}+F_N^{(2)},
\]
where $F_N^{(1)}$ contains the first $m_N$ columns of
$\frac{H_N}{(\log N)^\gamma}$ and zeros elsewhere, while $F_N^{(2)}$ contains the remaining
columns. Then
\[
\operatorname{rank}\left(F_N^{(1)}\right)\le m_N,
\]
and by~\eqref{eq:mN_log_cutoff}, we have
\[
\lim_{N\to\infty}
\frac{\operatorname{rank}\left(F_N^{(1)}\right)}{N}
\le
\lim_{N\to\infty}\frac{1}{(\log N)^\eta}=0 .
\]
Let $\varepsilon>0$ be fixed.
Set
\[ r=\operatorname{rank}\left(F_N^{(1)}\right), \quad
s=
\#\left\{k=1,\dots,N:\sigma_k\left(F_N^{(2)}\right)>\varepsilon\right\}.
\]
By Weyl's inequality, we get
\[
\sigma_{r+s+1}\left(F_N^{(1)}+F_N^{(2)}\right)
\le
\sigma_{r+1}\left(F_N^{(1)}\right)
+
\sigma_{s+1}\left(F_N^{(2)}\right)
\le
\varepsilon .
\]
Then, we obtain
\begin{equation}\label{ultfCT}
    \#\left\{
k=1,\dots,N:\sigma_k\left(\frac{H_N}{(\log N)^\gamma}\right)>\varepsilon
\right\}
\le
r
+
s.
\end{equation}
On the other hand, we have
\[
\#\left\{k=1,\dots,N:\sigma_k\left(F_N^{(2)}\right)>\varepsilon\right\}
\le
\frac{\left\|F_N^{(2)}\right\|_F^2}{\varepsilon^2}.
\]
Since $F_N^{(2)}$ contains the columns $m_N+1,\ldots,N$ of
$\frac{H_N}{(\log N)^\gamma}$, by~\eqref{eq:H_column_euclidean_bound}, we get
\begin{equation}\label{bound:BN}
    \left\|F_N^{(2)}\right\|_F^2
=
\frac{1}{(\log N)^{2\gamma}}
\sum_{j=m_N+1}^{N}
\left\|H_Ne_j\right\|_2^2
\le
\frac{C(\alpha,\beta,c_h)N}{(\log N)^{2\gamma}}
\sum_{j=m_N+1}^{N}\frac{1}{j}.
\end{equation}
From~\eqref{eq:mN_log_cutoff}, for $N$
sufficiently large, we have
\[
m_N
\ge
\frac{1}{2}\frac{N}{(\log N)^\eta},
\]
and then
\[
\frac{N}{m_N}\le 2(\log N)^\eta.
\]
Since $x\mapsto \frac{1}{x}$ is decreasing, we have
\[
\sum_{j=m_N+1}^{N}\frac{1}{j}
\le
\int_{m_N}^{N}\frac{dx}{x}
=
\log\frac{N}{m_N}
\le
\log\left(2(\log N)^\eta\right)
\le
C(\eta)\left(1+\log\log N\right).
\]
Consequently, by~\eqref{bound:BN}, we get
\[
\left\|F_N^{(2)}\right\|_F^2
\le
C(\alpha,\beta,c_h,\eta)
N
\frac{1+\log\log N}{(\log N)^{2\gamma}}.
\]
Finally, by~\eqref{ultfCT}, we obtain
\begin{eqnarray*}
    &&\frac{1}{N}
\#\left\{
k=1,\dots,N:\sigma_k\left(\frac{H_N}{(\log N)^\gamma}\right)>\varepsilon
\right\}\\
&\le&
\frac{1}{(\log N)^\eta}
+
\frac{C(\alpha,\beta,c_h,\eta)}{\varepsilon^2}
\frac{1+\log\log N}{(\log N)^{2\gamma}}.
\end{eqnarray*}
Since $\eta>0$ and $\gamma>0$, the right-hand side tends to zero as
$N\to\infty$. Therefore
\[
\lim_{N\to\infty}
\frac{\#\left\{
k=1,\dots,N:\sigma_k\left(\frac{H_N}{(\log N)^\gamma}\right)>\varepsilon
\right\}}{N}
=0
\]
for any $\varepsilon>0$. By Theorem~\ref{thmGaroni}, we conclude that
\[
\left\{\frac{H_N}{(\log N)^\gamma}\right\}_N\sim_{\mathrm{GLT}} 0.
\]
\end{proof}

\begin{remark}
Under the additional assumption that the mesh is quasi-uniform,
Proposition~\ref{prop:quant_scaled_RN_zero} remains valid also in the case
\[
\alpha,\beta\in\left(0,\frac{1}{2}\right).
\]
More precisely, assume that there exist constants $c_h,C_h>0$,
independent of $N$, such that
\begin{equation}\label{ChN}
\frac{c_h}{N}\le h_i\le \frac{C_h}{N},
\quad i=1,\dots,N.
\end{equation}
Indeed, by~\eqref{eq:rjtilde} and~\eqref{eq:weighted_poly_bound}, we have
\begin{eqnarray*}
\left|\widetilde r_j(x)\right|^2
&=&
\frac{4}{\left(j+\alpha+\beta+1\right)^2}
\left|P_j^{(\alpha,\beta)}(x)\omega_{\alpha,\beta}(x)\right|^2\\
&\le&
\frac{4}{\left(j+\alpha+\beta+1\right)^2}
K_j C_0(\alpha,\beta)
\left(1-x\right)^{\alpha-\frac{1}{2}}
\left(1+x\right)^{\beta-\frac{1}{2}}.
\end{eqnarray*}
Using also~\eqref{eq:Kj_bound}, and recalling that $\alpha$ and $\beta$ are
fixed, we obtain
\begin{equation}\label{eq121}
\left|\widetilde r_j(x)\right|^2
\le
\frac{C(\alpha,\beta)}{\left(j+1\right)^3}
\omega_{\alpha-\frac{1}{2},\beta-\frac{1}{2}}(x),
\quad x\in\left(-1,1\right).
\end{equation}
Since
\[
\alpha-\frac{1}{2}>-1,
\quad
\beta-\frac{1}{2}>-1,
\]
we have
\[
\omega_{\alpha-\frac{1}{2},\beta-\frac{1}{2}}\in L^1\left(-1,1\right).
\]
On the other hand, $\omega_{\alpha-\frac{1}{2},\beta-\frac{1}{2}}$ is not bounded at the endpoints. Therefore, in order to
estimate the discrete sum
\[
\sum_{k=1}^{N-1}\omega_{\alpha-\frac{1}{2},\beta-\frac{1}{2}}\left(x_k\right),
\]
we need to exploit the quasi-uniformity of the mesh. A direct computation gives
\[
\omega_{\alpha-\frac{1}{2},\beta-\frac{1}{2}}'\left(x\right)
=
\left(1-x\right)^{\alpha-\frac{3}{2}}
\left(1+x\right)^{\beta-\frac{3}{2}}
\left[
\left(\beta-\alpha\right)+\left(1-\alpha-\beta\right)x
\right].
\]
Since
\[
1-\alpha-\beta>0,
\]
the function $\omega_{\alpha-\frac{1}{2},\beta-\frac{1}{2}}$ has a unique critical point
\[
x^*=
\frac{\alpha-\beta}{1-\alpha-\beta}\in(-1,1).
\]
Therefore $\omega_{\alpha-\frac{1}{2},\beta-\frac{1}{2}}$ is decreasing on
\[
\left(-1,x^*\right]
\]
and increasing on
\[
\left[x^*,1\right).
\]
Then, for $N$ sufficiently large, there exists
$k^*\in\left\{1,\dots,N-2\right\}$ such that
\[
x_{k^*}\le x^*<x_{k^*+1}.
\]
Then
\begin{eqnarray*}
\sum_{k=1}^{N-1}\omega_{\alpha-\frac{1}{2},\beta-\frac{1}{2}}\left(x_k\right)
&=&
\omega_{\alpha-\frac{1}{2},\beta-\frac{1}{2}}\left(x_1\right)
+
\sum_{k=2}^{k^*}\omega_{\alpha-\frac{1}{2},\beta-\frac{1}{2}}\left(x_k\right)\\
&+&
\sum_{k=k^*+1}^{N-2}\omega_{\alpha-\frac{1}{2},\beta-\frac{1}{2}}\left(x_k\right)
+
\omega_{\alpha-\frac{1}{2},\beta-\frac{1}{2}}\left(x_{N-1}\right).
\end{eqnarray*}
We first estimate the two boundary terms. Since $x_0=-1$ and $x_N=1$, using
\[
h_1=x_1-x_0\le \frac{C_h}{N},
\quad
h_N=x_{N}-x_{N-1}\le \frac{C_h}{N},
\]
we have
\[
1-x_1=2-h_1\ge 2-\frac{C_h}{N},
\quad
1+x_{N-1}=2-h_N\ge 2-\frac{C_h}{N}.
\]
Hence, since
\[
\alpha-\frac{1}{2}<0, \quad \beta-\frac{1}{2}<0,
\]
we have
\[
\left(1-x_1\right)^{\alpha-\frac{1}{2}}\le C,
\quad
\left(1+x_{N-1}\right)^{\beta-\frac{1}{2}}\le C,
\]
with $C$ independent of $N$. Using also~\eqref{ChN}, we get
\begin{equation}\label{bound0}
\omega_{\alpha-\frac{1}{2},\beta-\frac{1}{2}}\left(x_1\right)
=
\left(1-x_1\right)^{\alpha-\frac{1}{2}}
\left(1+x_1\right)^{\beta-\frac{1}{2}}
\le
C h_1^{\beta-\frac{1}{2}}
\le
C N^{\frac{1}{2}-\beta}.
\end{equation}
Similarly, we have
\begin{eqnarray}\notag
\omega_{\alpha-\frac{1}{2},\beta-\frac{1}{2}}\left(x_{N-1}\right)
&=&
\left(1-x_{N-1}\right)^{\alpha-\frac{1}{2}}
\left(1+x_{N-1}\right)^{\beta-\frac{1}{2}}
\\&\le&
C h_N^{\alpha-\frac{1}{2}}
\le
C N^{\frac{1}{2}-\alpha}.\label{bound1}
\end{eqnarray}
Now, we consider the interior sums. Since $\omega_{\alpha-\frac{1}{2},\beta-\frac{1}{2}}$ is decreasing on $\left(-1,x^*\right],$ for any $k=2,\dots,k^*$ and any $t\in\left[x_{k-1},x_k\right]$,
we have
\[
\omega_{\alpha-\frac{1}{2},\beta-\frac{1}{2}}\left(x_k\right)\le \omega_{\alpha-\frac{1}{2},\beta-\frac{1}{2}}\left(t\right).
\]
Hence, integrating on $\left[x_{k-1},x_k\right]$, we get
\[
h_k \omega_{\alpha-\frac{1}{2},\beta-\frac{1}{2}}\left(x_k\right)
\le
\int_{x_{k-1}}^{x_k} \omega_{\alpha-\frac{1}{2},\beta-\frac{1}{2}}\left(t\right)dt.
\]
By~\eqref{ChN}, it follows that
\[
\omega_{\alpha-\frac{1}{2},\beta-\frac{1}{2}}\left(x_k\right)
\le
\frac{N}{c_h}
\int_{x_{k-1}}^{x_k} \omega_{\alpha-\frac{1}{2},\beta-\frac{1}{2}}\left(t\right)dt.
\]
Summing over $k=2,\dots,k^*$, we obtain
\begin{equation}\label{bound2}
\sum_{k=2}^{k^*}\omega_{\alpha-\frac{1}{2},\beta-\frac{1}{2}}\left(x_k\right)
\le
\frac{N}{c_h}
\int_{-1}^{1} \omega_{\alpha-\frac{1}{2},\beta-\frac{1}{2}}\left(t\right)dt.
\end{equation}
On the other hand, $\omega_{\alpha-\frac{1}{2},\beta-\frac{1}{2}}$ is increasing on $\left[x^*,1\right).$
Therefore, for any $k=k^*+1,\dots,N-2$ and any $t\in\left[x_k,x_{k+1}\right],$ we have
\[
\omega_{\alpha-\frac{1}{2},\beta-\frac{1}{2}}\left(x_k\right)\le \omega_{\alpha-\frac{1}{2},\beta-\frac{1}{2}}\left(t\right).
\]
Thus
\[
h_{k+1} \omega_{\alpha-\frac{1}{2},\beta-\frac{1}{2}}\left(x_k\right)
\le
\int_{x_k}^{x_{k+1}} \omega_{\alpha-\frac{1}{2},\beta-\frac{1}{2}}\left(t\right)dt,
\]
and, using again~\eqref{ChN}, we get
\[
\omega_{\alpha-\frac{1}{2},\beta-\frac{1}{2}}\left(x_k\right)
\le
\frac{N}{c_h}
\int_{x_k}^{x_{k+1}} \omega_{\alpha-\frac{1}{2},\beta-\frac{1}{2}}\left(t\right)dt.
\]
Summing over $k=k^*+1,\dots,N-2$, we obtain
\begin{equation}\label{bound3}
\sum_{k=k^*+1}^{N-2}\omega_{\alpha-\frac{1}{2},\beta-\frac{1}{2}}\left(x_k\right)
\le
\frac{N}{c_h}
\int_{-1}^{1} \omega_{\alpha-\frac{1}{2},\beta-\frac{1}{2}}\left(t\right)dt.
\end{equation}
Combining~\eqref{bound0},~\eqref{bound1},~\eqref{bound2}, and~\eqref{bound3},
we conclude that
\[
\sum_{k=1}^{N-1}\omega_{\alpha-\frac{1}{2},\beta-\frac{1}{2}}\left(x_k\right)
\le
C N.
\]
Finally, since
\[
\widetilde r_j\left(-1\right)=\widetilde r_j\left(1\right)=0,
\]
by~\eqref{eq121} we obtain for $j=1,\dots,N$
\begin{eqnarray*}
\left\|r_j^{(N)}\right\|^2
=
\sum_{k=1}^{N-1}\left|\widetilde r_j\left(x_k\right)\right|^2
&\le&
\frac{C(\alpha,\beta)}{\left(j+1\right)^3}
\sum_{k=1}^{N-1}\omega_{\alpha-\frac{1}{2},\beta-\frac{1}{2}}\left(x_k\right)\\
&\le&
C_1(\alpha,\beta)\frac{N}{\left(j+1\right)^3}.
\end{eqnarray*}
This is the estimate needed in the proofs of
Proposition~\ref{prop:quant_scaled_RN_zero}. Hence, once this bound has been
established, the same strategy applies.
\end{remark}

\subsection{Stability estimates}
Now we study the stability of the weighted Jacobi histopolation operator. More precisely, we derive bounds that hold for arbitrary meshes and for any discretization level $N$. Let $X_N=\left\{x_0,\dots,x_N\right\}$ be the node set satisfying~\eqref{nodecond}, and let $\omega_{\alpha,\beta}$ be the Jacobi weight defined in~\eqref{jacweight}. In the following, we assume that
\[
\alpha,\beta>-1/2.
\]
For the purposes of this section, we express the polynomial histopolant~\eqref{pol_hist} in the Jacobi basis $\left\{P_{j-1}^{(\alpha,\beta)}\right\}_{j=1}^N$, instead of the basis used in the previous section. Thus, we write
\begin{equation}\label{polexp}
    p_{N-1}(x)=\sum_{j=1}^{N}c_j P_{j-1}^{(\alpha,\beta)}(x),
\quad
\boldsymbol{c}=\left[c_1,\dots,c_N\right]^\top\in\mathbb{R}^N.
\end{equation}
With a slight abuse of notation, we denote by $H_N\in\mathbb{R}^{N\times N}$ the histopolation matrix corresponding to this choice of basis, namely
\begin{equation*}
    \left[H_N\right]_{i,j}=\frac{1}{h_i}\int_{s_i} P_{j-1}^{(\alpha,\beta)}(x) \omega_{\alpha,\beta}(x)dx,
\quad s_i=\left[x_{i-1},x_i\right], \quad i,j=1,\dots,N.
\end{equation*}
For any $\boldsymbol v=[v_1,\dots,v_N]^\top\in\mathbb{R}^N$, we introduce the mesh-weighted discrete norm
\begin{equation}\label{eq:weigh_norm}
    \|\boldsymbol v\|_h^2=\sum_{i=1}^N h_i \left|v_i\right|^2.
\end{equation}
Finally, we define the Gram matrix $\widetilde G\in\mathbb{R}^{N\times N}$ by
\begin{equation}\label{eq:matGtilde}
    \left[\widetilde G\right]_{\ell,k}
=
\int_{-1}^1
P_{\ell-1}^{(\alpha,\beta)}(x)
P_{k-1}^{(\alpha,\beta)}(x)
\left(\omega_{\alpha,\beta}(x)\right)^2dx,
\quad \ell,k=1,\dots,N .
\end{equation}

\begin{remark}
The matrix $\widetilde G$ is well defined. Indeed, since
\[
\left(\omega_{\alpha,\beta}(x)\right)^2=(1-x)^{2\alpha}(1+x)^{2\beta},
\]
and $\alpha,\beta>-\frac{1}{2}$, we have
\[
2\alpha>-1,
\quad
2\beta>-1.
\]
Therefore $\left(\omega_{\alpha,\beta}(x)\right)^2\in L^1(-1,1)$, and all the entries of $\widetilde G$ are finite. Moreover, $\widetilde G$ is the Gram matrix of the linearly independent family
\[
\left\{P_{j-1}^{(\alpha,\beta)}\right\}_{j=1}^N.
\]
Consequently, $\widetilde G$ is symmetric positive definite.
\end{remark}

The next result provides the basic stability estimate for the histopolation matrix in the mesh-weighted discrete norm.
\begin{theorem}
\label{thm:uniform_interior_stability}
For any $\boldsymbol c\in\mathbb R^N$, the following inequality holds
\begin{equation*}
    \left\|H_N \boldsymbol{c}\right\|_{h}^2
\le \|\boldsymbol{c}\|_{2}^2
\lambda_{\max}\left(\widetilde G\right),
\end{equation*}
where  $\lambda_{\max}\left(\widetilde G\right)$ denotes the maximum eigenvalue of
$\widetilde G$. In particular
\[
\left\|H_N \boldsymbol{c}\right\|_{h}
\le \|\boldsymbol{c}\|_{2}
\sqrt{\lambda_{\max}\left(\widetilde G\right)}.
\]
\end{theorem}

\begin{proof}
Let
\begin{equation*}
    b_i=\left[H_N \boldsymbol{c}\right]_i=\frac{1}{h_i}\int_{s_i} p_{N-1}(x)\omega_{\alpha,\beta}(x)dx,
\quad i=1,\dots,N,
\end{equation*}
and set $\boldsymbol{b}=\left[b_1,\dots,b_N\right]^\top$. Multiplying by $h_i$, we get
\[
h_i b_i=\int_{s_i} p_{N-1}(x)\omega_{\alpha,\beta}(x)dx.
\]
Hence, by the Cauchy--Schwarz inequality
\begin{eqnarray*}
    \left|h_i b_i\right|^2
&=&
\left|\int_{s_i} p_{N-1}(x)\omega_{\alpha,\beta}(x)dx\right|^2\\
&\le& \left(\int_{s_i}1^2dx\right)
\left(\int_{s_i}\left|p_{N-1}(x)\right|^2\left(\omega_{\alpha,\beta}(x)\right)^2dx\right)\\&=&h_i\int_{s_i}\left|p_{N-1}(x)\right|^2\left(\omega_{\alpha,\beta}(x)\right)^2dx.
\end{eqnarray*}
Dividing by $h_i$, we have
\[
h_i\left|b_i\right|^2
\le
\int_{s_i}\left|p_{N-1}(x)\right|^2\left(\omega_{\alpha,\beta}(x)\right)^2dx,
\quad i=1,\dots,N.
\]
Since the cells $\left\{s_i\right\}_{i=1}^N$ form a partition of $[-1,1]$, summing over $i=1,\dots,N$, we obtain
\begin{eqnarray*}
    \sum_{i=1}^{N} h_i\left|b_i\right|^2
&\le&
\sum_{i=1}^{N}\int_{s_i}\left|p_{N-1}(x)\right|^2\left(\omega_{\alpha,\beta}(x)\right)^2dx\\ &=&\int_{-1}^{1}\left|p_{N-1}(x)\right|^2\left(\omega_{\alpha,\beta}(x)\right)^2dx.
\end{eqnarray*}
Therefore, by~\eqref{eq:weigh_norm}, we get
\begin{equation}\label{kq1_new}
\left\|H_N \boldsymbol{c}\right\|_{h}^2
=
\left\|\boldsymbol{b}\right\|_{h}^2
\le
\int_{-1}^{1}\left|p_{N-1}(x)\right|^2\left(\omega_{\alpha,\beta}(x)\right)^2dx.
\end{equation}
Using~\eqref{polexp} and the definition of $\widetilde G$, we have
\begin{equation}\label{kq2_new}
\int_{-1}^{1}\left|p_{N-1}(x)\right|^2\left(\omega_{\alpha,\beta}(x)\right)^2dx
=
\boldsymbol{c}^\top \widetilde G \boldsymbol{c}.
\end{equation}
Hence, combining~\eqref{kq1_new} and~\eqref{kq2_new}, we get
\begin{equation*}
    \left\|H_N \boldsymbol{c}\right\|_{h}^2
\le
\boldsymbol{c}^\top \widetilde G \boldsymbol{c}.
\end{equation*}
Since $\widetilde G$ is symmetric positive definite, the Rayleigh quotient yields
\begin{equation*}
\boldsymbol{c}^\top \widetilde G\boldsymbol{c}
\le\|\boldsymbol{c}\|_{2}^2
\lambda_{\max}\left(\widetilde G\right).
\end{equation*}
\end{proof}

\begin{remark}
By Theorem~\ref{thm:uniform_interior_stability}, we have
\begin{equation*}
    \left\|H_N \boldsymbol c\right\|_h^2 \le \boldsymbol c^\top \widetilde G \boldsymbol c,
\quad \boldsymbol c\in\mathbb R^N.
\end{equation*}
This relation is equivalent to
\[
H_N^\top D_h H_N \preceq \widetilde G,
\]
where $D_h\in \mathbb{R}^{N\times N}$ is defined as
\begin{equation*}
D_h=\operatorname{diag}\left(h_1,\ldots,h_N\right),
\quad
h_i=x_i-x_{i-1}.
\end{equation*}
With this definition, the mesh-weighted discrete norm in~\eqref{eq:weigh_norm}
can be written as
\[
\|\boldsymbol{v}\|_h^2=\boldsymbol{v}^{\top}D_h \boldsymbol{v}.
\]
Hence, for any $\boldsymbol c\in\mathbb R^N$, we have
\[
\left\|H_N \boldsymbol c\right\|_h^2
=
\left(H_N\boldsymbol c\right)^\top D_h \left(H_N\boldsymbol c\right)
=
\boldsymbol c^\top H_N^\top D_h H_N \boldsymbol c.
\]
Therefore
\[
\left\|H_N \boldsymbol c\right\|_h^2 \le \boldsymbol c^\top \widetilde G \boldsymbol c
\quad \text{for any } \boldsymbol c\in\mathbb R^N
\]
if and only if
\[
\boldsymbol c^\top H_N^\top D_h H_N \boldsymbol c
\le
\boldsymbol c^\top \widetilde G \boldsymbol c
\quad \text{for any } \boldsymbol c\in\mathbb R^N.
\]
Since both $H_N^\top D_h H_N$ and $\widetilde G$ are symmetric, this is equivalent to
\[
H_N^\top D_h H_N \preceq \widetilde G.
\]
\end{remark}
In order to study $\lambda_{\max}\left(\widetilde G\right)$, we now write
\begin{equation*}
    \widetilde G_N := \widetilde G \in \mathbb{R}^{N\times N}.
\end{equation*}

\begin{theorem}
\label{prop:lmax_Gtilde}
There exists a constant $C(\alpha,\beta)>0$, depending only on $\alpha$ and $\beta$, such that
\[
\lambda_{\max}\left(\widetilde G_N\right)
\le
C(\alpha,\beta)\left(1+\log N\right),
\quad N\ge1.
\]
\end{theorem}

\begin{proof}
Since $\widetilde G_N$ is symmetric positive definite,  by~\eqref{eq:matGtilde}, we have
\begin{equation}\label{matGtrace}
    \lambda_{\max}\left(\widetilde G_N\right)
\le
\operatorname{tr}\left(\widetilde G_N\right)
=
\sum_{j=0}^{N-1}
\int_{-1}^{1}
\left|P_j^{(\alpha,\beta)}(x)\right|^2
\left(\omega_{\alpha,\beta}(x)\right)^2dx.
\end{equation}
Therefore, it is enough to prove that there exists a constant $C(\alpha,\beta)>0$ such that
\begin{equation}\label{eq:diag_estimate_Gtilde}
\int_{-1}^{1}
\left|P_j^{(\alpha,\beta)}(x)\right|^2
\left(\omega_{\alpha,\beta}(x)\right)^2dx
\le
\frac{C(\alpha,\beta)}{j+1},
\quad j\ge0.
\end{equation}
Since $\alpha,\beta>-1/2$, the weighted estimate for orthonormal Jacobi
polynomials proved in~\cite{Nevai:1994:GJW}, applied to the family
$\left\{\widehat P_j^{(\alpha,\beta)}\right\}_{j\ge0}$ defined in~\eqref{ortpol}, leads to
\begin{equation*}
(1-x)^{\alpha+\frac{1}{2}}(1+x)^{\beta+\frac{1}{2}}
\left|\widehat{P}_j^{(\alpha,\beta)}(x)\right|^2
\le
C_2(\alpha,\beta),
\quad x\in(-1,1),\quad j\ge0,
\end{equation*}
where
\[
C_2(\alpha,\beta)=\frac{2e\left(2+\sqrt{\alpha^2+\beta^2}\right)}{\pi}.
\]
Hence
\[
\left|\widehat{P}_j^{(\alpha,\beta)}(x)\right|^2
\le C_2(\alpha,\beta)
(1-x)^{-\alpha-\frac{1}{2}}(1+x)^{-\beta-\frac{1}{2}}.
\]
Multiplying the previous inequality by $\left(\omega_{\alpha,\beta}(x)\right)^2$, we obtain
\[
\left|\widehat{P}_j^{(\alpha,\beta)}(x)\right|^2\left(\omega_{\alpha,\beta}(x)\right)^2
\le
C_2(\alpha,\beta)
(1-x)^{\alpha-\frac{1}{2}}(1+x)^{\beta-\frac{1}{2}},
\]
for any $x\in(-1,1)$, and $j\ge0$.
Since $\alpha,\beta>-1/2$, we have
\[
\alpha-\frac{1}{2}>-1,
\quad
\beta-\frac{1}{2}>-1,
\]
and hence
\[
(1-x)^{\alpha-\frac{1}{2}}(1+x)^{\beta-\frac{1}{2}}\in L^1(-1,1).
\]
Therefore, there exists a constant $C_3(\alpha,\beta)>0$ such that
\[
\int_{-1}^{1}
\left|\widehat{P}_j^{(\alpha,\beta)}(x)\right|^2
\left(\omega_{\alpha,\beta}(x)\right)^2dx\le
C_3(\alpha,\beta),
\]
for any $j\ge0.$
Then, using~\eqref{ortpol} and~\eqref{eq:Kj_bound}, we get
\begin{eqnarray*}
    \int_{-1}^{1}
\left|P_j^{(\alpha,\beta)}(x)\right|^2
\left(\omega_{\alpha,\beta}(x)\right)^2dx
&=&
K_j
\int_{-1}^{1}
\left|\widehat{P}^{(\alpha,\beta)}_j(x)\right|^2\left(\omega_{\alpha,\beta}(x)\right)^2dx\\
&\le&
\frac{C(\alpha,\beta)}{j+1},
\quad j\ge0,
\end{eqnarray*}
which proves~\eqref{eq:diag_estimate_Gtilde}.

Since the function $x\mapsto 1/x$ is decreasing on $[1,+\infty)$, we have
\[
\sum_{j=0}^{N-1}\frac{1}{j+1}
=
\sum_{k=1}^{N}\frac{1}{k}
=
1+\sum_{k=2}^{N}\frac{1}{k}
\le
1+\int_{1}^{N}\frac{1}{x}dx
=
1+\log N.
\]
Thus, by~\eqref{matGtrace} and~\eqref{eq:diag_estimate_Gtilde}, we obtain
\[
\lambda_{\max}\left(\widetilde G_N\right)
\le
\operatorname{tr}\left(\widetilde G_N\right)
\le
C(\alpha,\beta)\sum_{j=0}^{N-1}\frac{1}{j+1}
\le
C(\alpha,\beta)\left(1+\log N\right).
\]
The proof is complete.
\end{proof}

\begin{remark}
The previous result immediately leads to a bound for the operator norm of the weighted histopolation matrix $H_N$, viewed as a linear map
\[
H_N : \left(\mathbb{R}^N,\|\cdot\|_2\right) \to \left(\mathbb{R}^N,\|\cdot\|_{h}\right),
\]
where $\|\cdot\|_{h}$ is the norm defined in~\eqref{eq:weigh_norm}. Indeed, by Theorem~\ref{thm:uniform_interior_stability}, we have
\[
\left\|H_N\right\|_{2\to h}
=
\sup_{\boldsymbol{c}\neq 0}
\frac{\left\|H_N\boldsymbol{c}\right\|_{h}}{\|\boldsymbol{c}\|_2}
\le
\sqrt{\lambda_{\max}\left(\widetilde G_N\right)}.
\]
Combining this with Theorem~\ref{prop:lmax_Gtilde}, we obtain
\[
\left\|H_N\right\|_{2\to h}
\le
\sqrt{C(\alpha,\beta)\left(1+\log N\right)}.
\]
 In particular, the operator norm of $H_N$ grows at most logarithmically.
\end{remark}

\section{Numerical results}\label{sec5}

In this section, we present numerical experiments supporting the asymptotic
analysis developed in the previous sections. We first study the decay,
under different normalizations, of the proportion of singular values larger
than a prescribed threshold. We then consider the unscaled histopolation
matrices, and finally compare the singular values of the nontrivial factors
appearing in the GLT analysis with the corresponding symbols, which have been deduced in our theoretical analysis.

All computations are performed in MATLAB. The entries of the histopolation
matrices are computed by Gauss--Legendre quadrature, using the quadrature
nodes and weights generated by the command \texttt{legpts}.

We start with the uniform partition of $[-1,1]$, with
\[
N=1000k, \quad k=1,\ldots,10.
\]
For each $N$, we compute the singular values of $H_N$ and consider the four
scaled matrices
\[
H_N^{(1)}=\frac{H_N}{N}, \quad
H_N^{(2)}=\frac{H_N}{N^{9/10}}, \quad
H_N^{(3)}=\frac{H_N}{N^{4/5}}, \quad
H_N^{(4)}=\frac{H_N}{(\log N)^4}.
\]
For a fixed threshold $\varepsilon>0$, we define
\[
q_N^{(i)}(\varepsilon)
=
\frac{\#\left\{j:\sigma_j\left(H_N^{(i)}\right)>\varepsilon\right\}}{N}, \quad  i=1,\ldots,4.
\]
Thus, $q_N^{(i)}(\varepsilon)$ measures the fraction of singular values of the
$i$-th scaled matrix which remain larger than $\varepsilon$. In the
experiments we use the thresholds
\[
\varepsilon=10^{-2}, \qquad
\varepsilon=5\cdot 10^{-3}, \qquad
\varepsilon=10^{-3}.
\]
We first consider the symmetric case $\alpha=\beta=2$. The results are
reported in Figures~\ref{fig:alpha2beta2_first}
and~\ref{fig:alpha2beta2_second}. In all four scalings, and for all the
thresholds considered, the quantity $q_N^{(i)}(\varepsilon)$ decreases as
$N$ increases. The decay rate depends on the scaling, while the overall
behaviour is consistent with the zero-distribution result proved above.
\begin{figure}[t]
\centering
\includegraphics[width=0.49\textwidth]{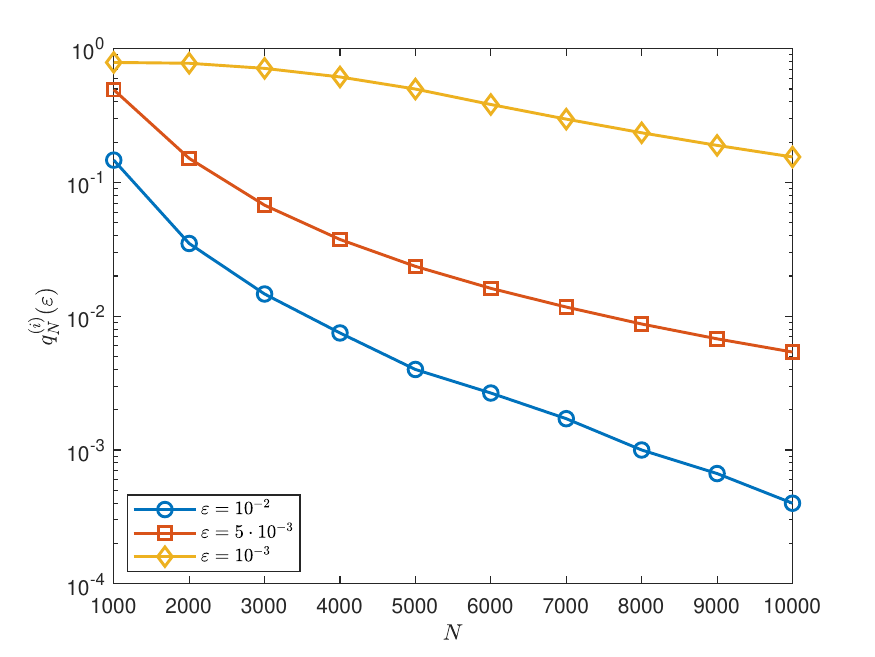}
\includegraphics[width=0.49\textwidth]{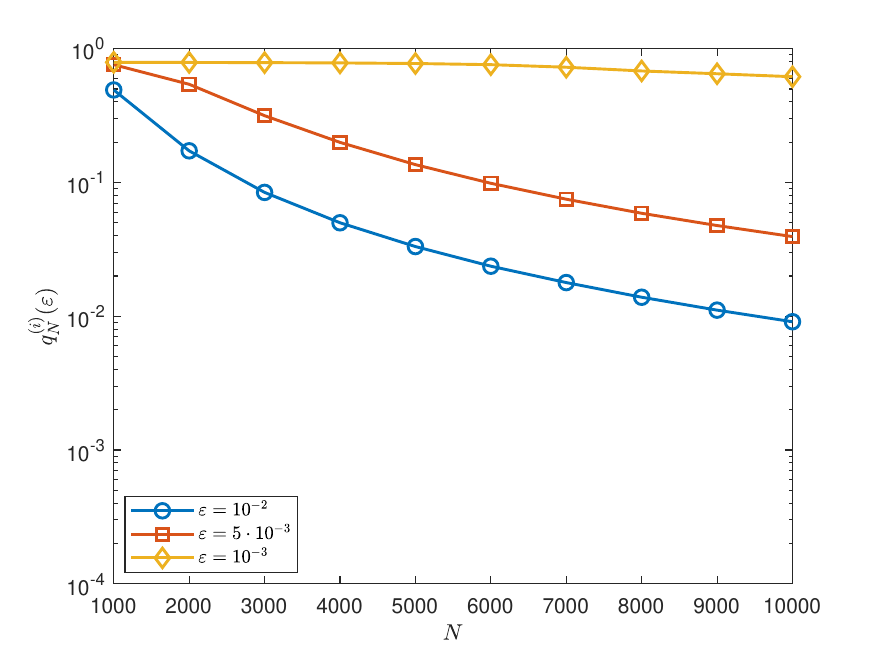}
\caption{Behaviour of $q_N^{(i)}(\varepsilon)$ for the scaled matrices
$H_N^{(1)}$ (left) and $H_N^{(2)}$ (right), with
$N=1000k$, $k=1,\ldots,10$, in the case $\alpha=\beta=2$.}
\label{fig:alpha2beta2_first}
\end{figure}

\begin{figure}[t]
\centering
\includegraphics[width=0.49\textwidth]{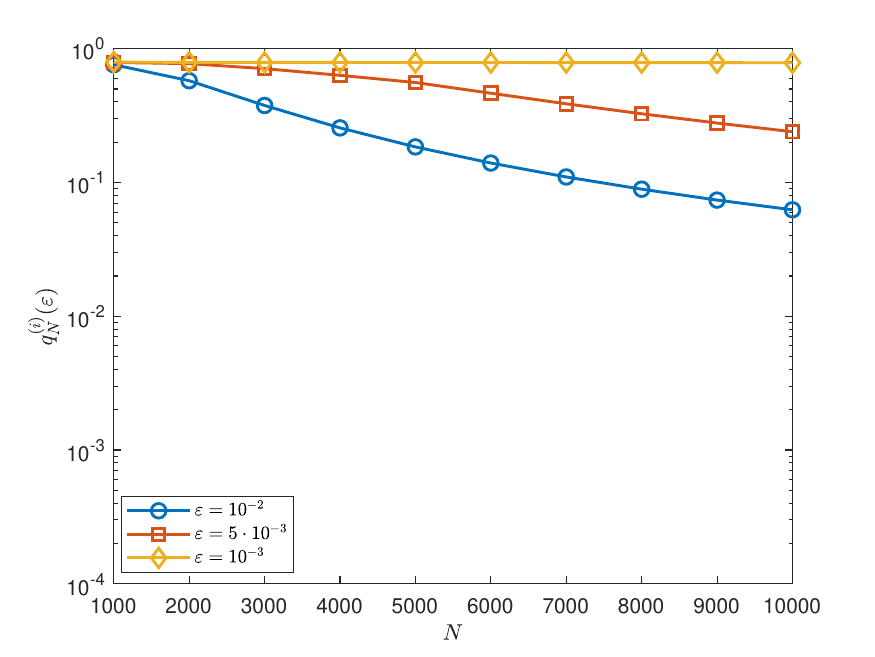}
\includegraphics[width=0.49\textwidth]{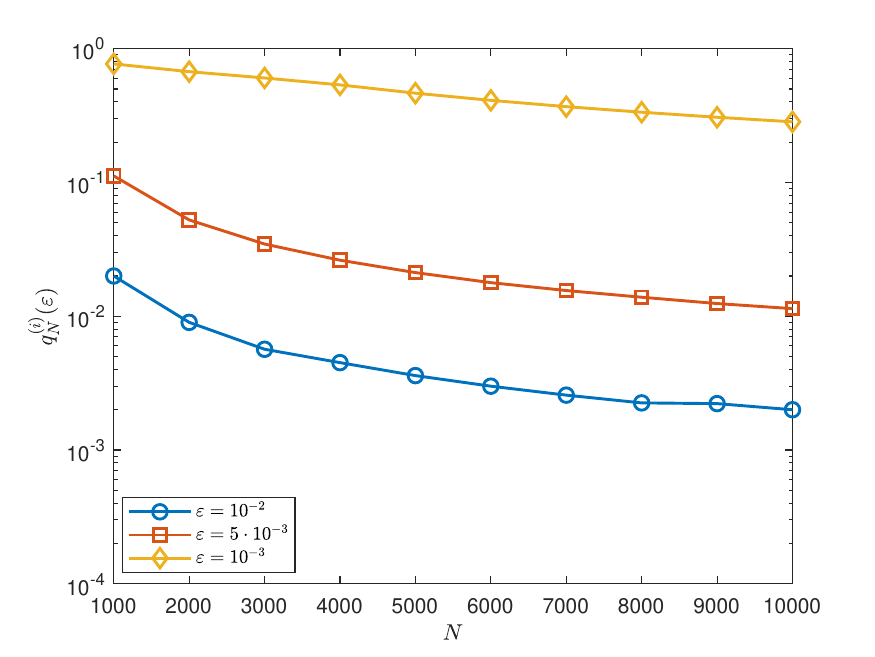}
\caption{Behaviour of $q_N^{(i)}(\varepsilon)$ for the scaled matrices
$H_N^{(3)}$ (left) and $H_N^{(4)}$ (right), with
$N=1000k$, $k=1,\ldots,10$, in the case $\alpha=\beta=2$.}
\label{fig:alpha2beta2_second}
\end{figure}

The same test is then repeated for the nonsymmetric choice
$\alpha=3/2$ and $\beta=1$. The results are reported in
Figures~\ref{fig:alpha32beta1_first}
and~\ref{fig:alpha32beta1_second}. The same qualitative behaviour is observed
in this case. For all the scalings and for all the thresholds considered, the
fraction of singular values greater than the threshold decreases as $N$
increases. This suggests that the observed behaviour is not tied to the
symmetric choice of the Jacobi parameters, but reflects the structure of the
weighted histopolation matrices under study.
\begin{figure}[t]
\centering
\includegraphics[width=0.49\textwidth]{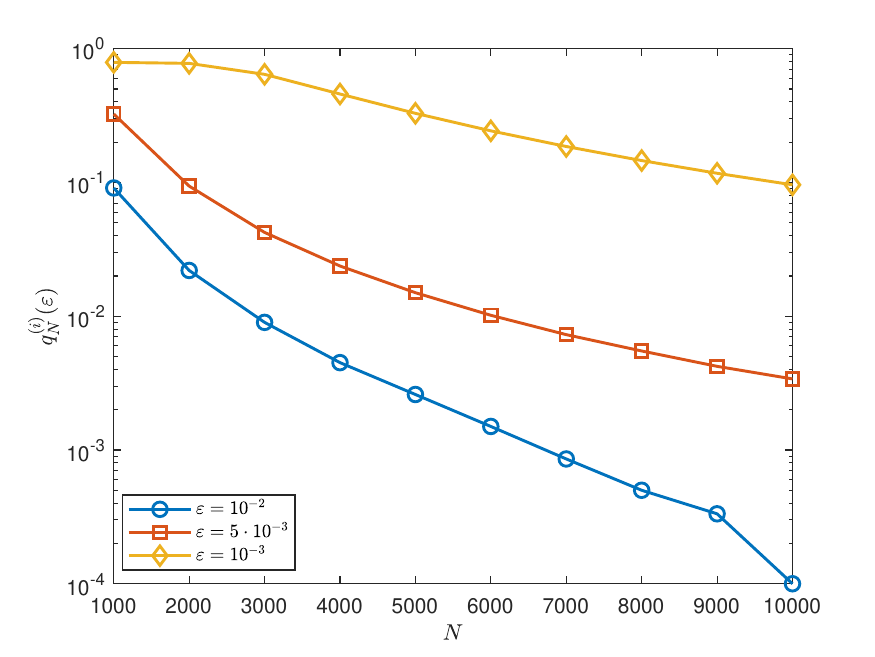}
\includegraphics[width=0.49\textwidth]{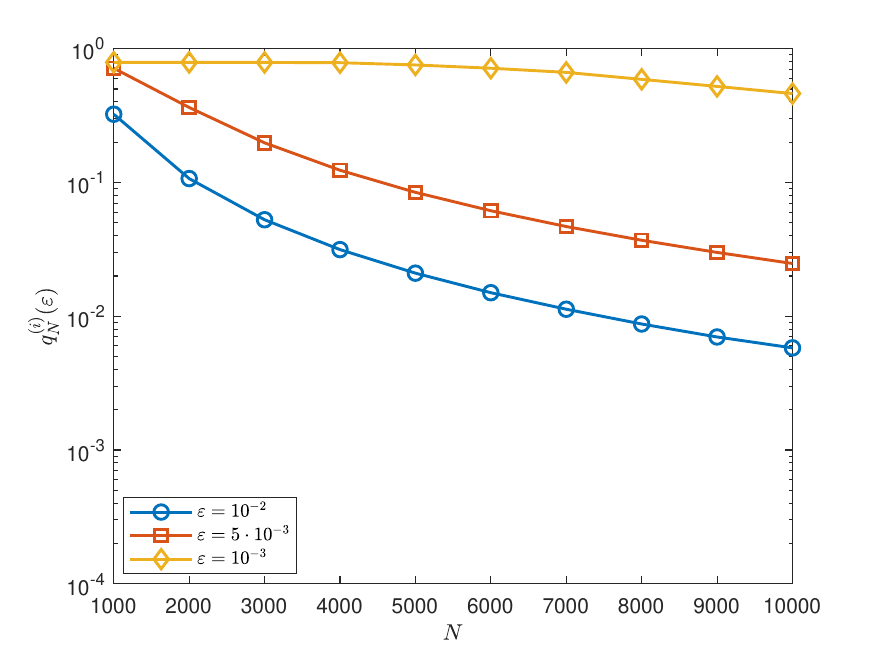}
\caption{Behaviour of $q_N^{(i)}(\varepsilon)$ for the scaled matrices
$H_N^{(1)}$ (left) and $H_N^{(2)}$ (right), with
$N=1000k$, $k=1,\ldots,10$, in the case $\alpha=3/2$ and $\beta=1$.}
\label{fig:alpha32beta1_first}
\end{figure}

\begin{figure}[t]
\centering
\includegraphics[width=0.49\textwidth]{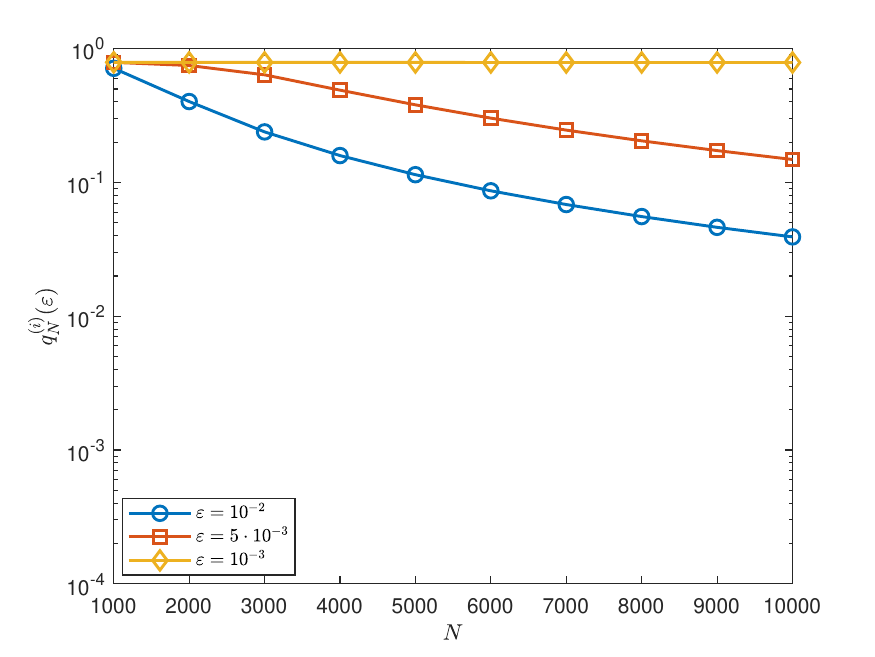}
\includegraphics[width=0.49\textwidth]{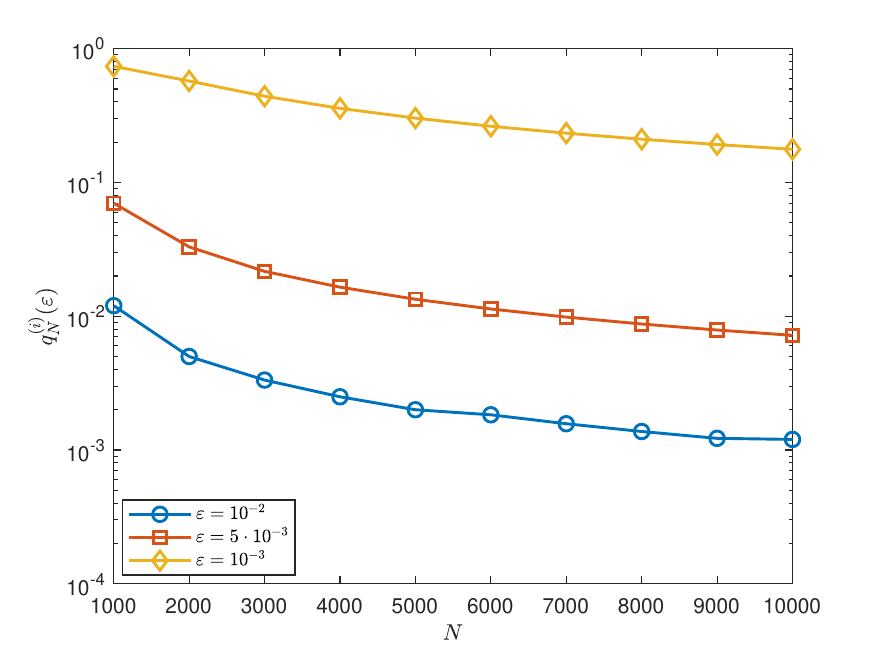}
\caption{Behaviour of $q_N^{(i)}(\varepsilon)$ for the scaled matrices
$H_N^{(3)}$ (left) and $H_N^{(4)}$ (right), with
$N=1000k$, $k=1,\ldots,10$, in the case $\alpha=3/2$ and $\beta=1$.}
\label{fig:alpha32beta1_second}
\end{figure}

For completeness, we also show the singular values of the unscaled
histopolation matrices $H_N$ for different values of $N$. We consider two
families of nonuniform nodes
\[
x_i=-1+2g(i/N), \quad i=0,\ldots,N,
\]
generated by
\begin{equation} \label{numtestg}
    g(y)=\frac{e^y-1}{e-1}
\qquad\text{and}\quad
g(y)=y^2.
\end{equation}
The results are shown in Figure~\ref{fig:HN_unscaled_exp_square}.

\begin{figure}[t]
\centering
\includegraphics[width=0.49\textwidth]{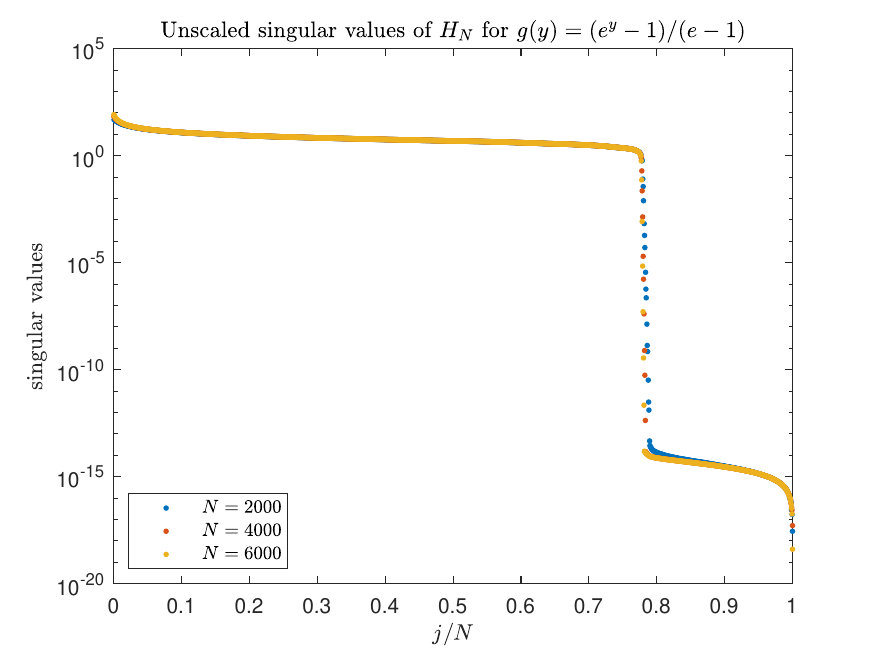}
\includegraphics[width=0.49\textwidth]{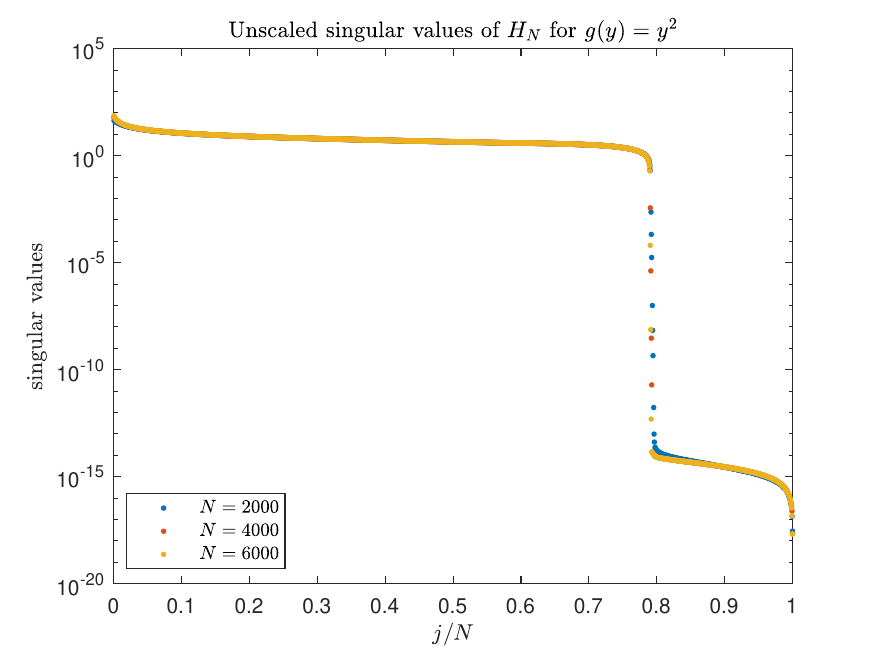}
\caption{Unscaled singular values of $H_N$ for the nodes
$x_i=-1+2g(i/N)$, with $g(y)=(e^y-1)/(e-1)$ (left) and $g(y)=y^2$ (right).}
\label{fig:HN_unscaled_exp_square}
\end{figure}

\begin{remark}\label{conj-glt}
From the plots in Figure \ref{fig:HN_unscaled_exp_square}, we observe that, for most values of $j/N$, the singular
values of $H_N$ seem to follow a well-defined curve. In the final part of the
spectrum, approximately for $j/N\in[0.8,1]$, they rapidly approach zero. This
suggests that the unscaled sequence $\{H_N\}_N$ may admit a singular value
distribution in accordance with Definition \ref{def:sv-ev-distribution}. In particular, from the graph, it can be conjectured that the symbol (if it exists) is identically zero in a set of measure $\frac{1}{5}\mu_k(D)$, where $k=2$ and the canonical choice of the domain is $D=[0,1]\times [-\pi,\pi]$ if $\{H_N\}_N$ has a GLT nature: we leave these refinements of the GLT analysis for future work.
\end{remark}

We close this section by comparing the singular values of the two nontrivial
factors appearing in the GLT analysis with the corresponding symbols. The
comparison is made at the level of singular value distributions. Accordingly,
the singular values are sorted in decreasing order and plotted against the
normalized index. The same idea is used for the symbols. More precisely, we
sample the modulus of the symbol on a uniform tensor grid in the variables
$(y,\theta)$, sort the resulting values in decreasing order, and plot the
corresponding nonincreasing rearrangement on $[0,1]$; see Remark \ref{rearrangements}. 

We first consider the banded matrix $T_N^{(J)}$. According to the analysis above, the symbol of the scaled sequence
$\{N T_N^{(J)}\}_N$ is the function defined in~\eqref{funkappa}. In Figure~\ref{fig:NTJ_symbol}, we show the comparison
for 
\[
N=2000, \quad \text{and} \quad \alpha=\beta=2.
\]
The circles denote the singular values of $N T_N^{(J)}$, and the continuous
curve represents the sorted samples of the modulus of the corresponding symbol.
Their close agreement is consistent with the singular value distribution
predicted by the GLT analysis.

\begin{figure}[t]
\centering
\includegraphics[width=0.49\textwidth]{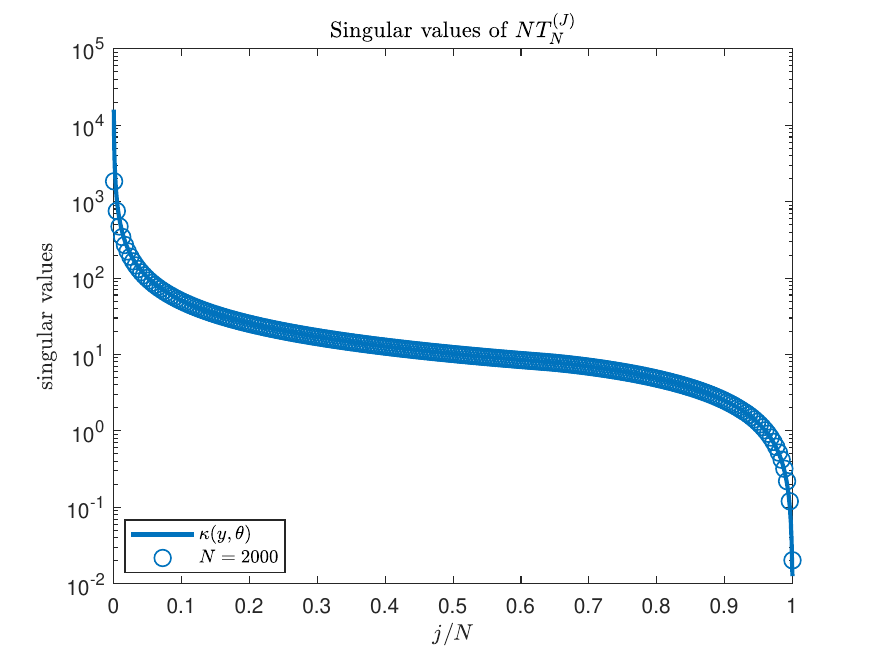}
\caption{Singular values of $N T_N^{(J)}$ for $N=2000$ compared with the
sorted samples of the modulus of its GLT symbol.}
\label{fig:NTJ_symbol}
\end{figure}

We next consider the backward-difference factor. The nodes are again chosen as
\[
x_i=-1+2g(i/N), \qquad i=0,\ldots,N,
\]
with $g$ given in~\eqref{numtestg}. With this convention, the symbol of
$\Delta_N/N$ is
\[
\kappa(y,\theta)
=
\frac{1-e^{i\theta}}{2g'(y)}.
\]
In Figure~\ref{fig:Delta_exp_square_symbol}, we report the comparison for
$N=2000$. In both cases, the sorted singular values of $\Delta_N/N$ closely
follow the sorted samples of $|\kappa(y,\theta)|$, in accordance with the GLT symbol.
\begin{figure}[t]
\centering
\includegraphics[width=0.49\textwidth]{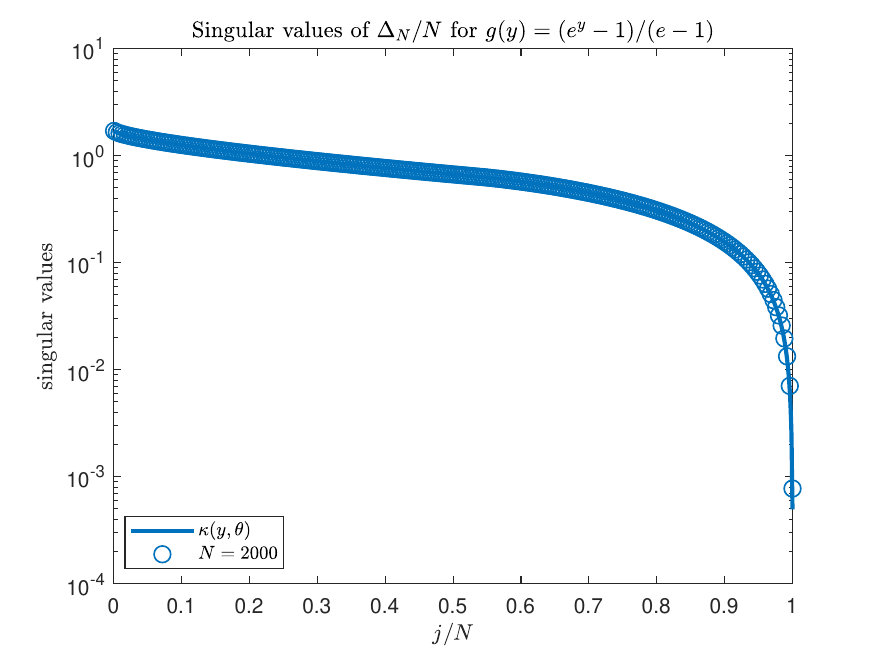}
\includegraphics[width=0.49\textwidth]{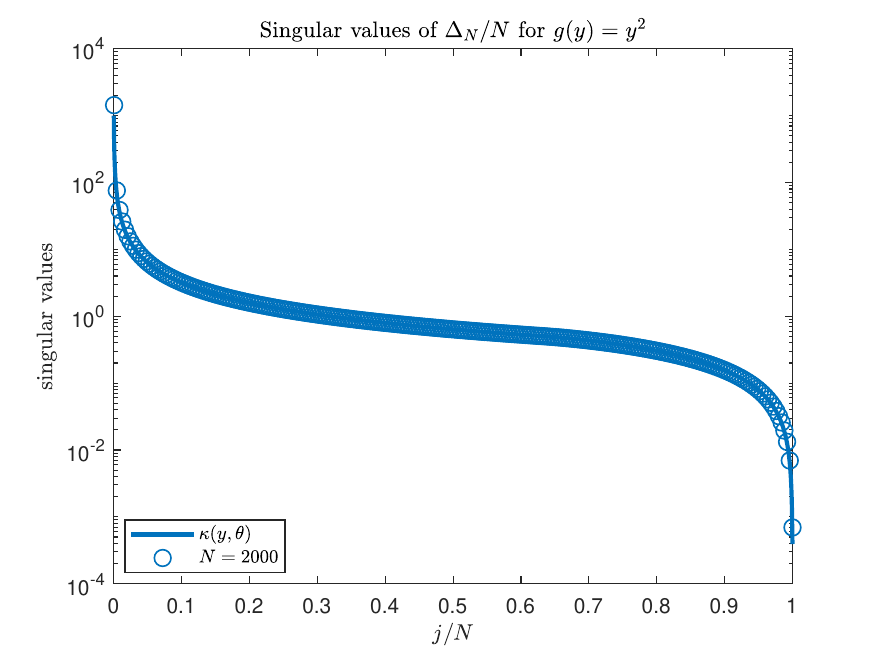}
\caption{Singular values of $\Delta_N/N$ for the nodes
$x_i=-1+2g(i/N)$, with $g(y)=(e^y-1)/(e-1)$ (left) and $g(y)=y^2$ (right).}
\label{fig:Delta_exp_square_symbol}
\end{figure}

Overall, the numerical experiments support the asymptotic framework developed in
the paper. The first set of tests illustrates, under the prescribed scalings, the
decay of the fraction of singular values of the histopolation matrices exceeding
a fixed positive threshold. The plots for the unscaled matrices provide further
numerical insight into the sequence $\{H_N\}_N$ and suggest a possible refinement
of the present analysis. Finally, the comparisons with the sampled symbols give
a direct visualization of the singular value distributions associated with the
main factors in the GLT description.

\section{Conclusions and Future Work}\label{sec6}
In this paper we have studied a weighted histopolation problem on $[-1,1]$
associated with Jacobi weights. Starting from weighted cell averages, we
introduced a reconstruction procedure in which the approximation space is
described through a Jacobi-type basis. This choice is natural for the analysis,
since it leads to weighted primitives whose structure can be handled explicitly.
In particular, the recurrence properties of Jacobi polynomials allow one to
express the primitive sampling matrix in terms of a banded spectral coupling,
thereby providing the link between the histopolation problem and the asymptotic
matrix analysis. We then studied the matrix sequences generated by this construction within the
GLT framework. Under the assumptions on the mesh and on the Jacobi parameters
considered in the paper, this analysis describes the singular value behaviour of
the relevant sequences and shows that the appropriate scaled histopolation
matrices are zero-distributed. This provides a theoretical explanation for the
decay observed in the numerical experiments. The stability discussion was carried out for the same weighted histopolation
problem written in the standard Jacobi basis. In this form, the estimates are
naturally expressed in a mesh-weighted discrete norm and show that the
corresponding histopolation operator exhibits only logarithmic growth with
respect to the discretization size. This complements the GLT analysis by
providing a finite-dimensional stability bound valid at every discretization
level.

Several questions remain open. A first direction is to extend the present
analysis to more general families of orthogonal polynomials, in order to
understand which parts of the structure found here are specific to Jacobi
weights and which persist in a broader setting. In accordance with the numerical experiments and the related statements 
conjectured in Remark \ref{conj-glt}, a second problem concerns a sharper asymptotic description of the unscaled histopolation matrices
$\{H_N\}_N$, and in particular determining the exact GLT symbol of the corresponding matrix sequence. 
Finally, it would be important to prove inverse-type inequalities for $H_N$, under suitable
regularity assumptions on the mesh. Such estimates would complement the
stability bounds obtained here and would give a more complete description of
the conditioning of the weighted histopolation operator.

\section*{Declarations}

\textbf{Corresponding author}\\
Federico Nudo, email federico.nudo@unical.it\\

\noindent
\textbf{Conflict of Interest}\\
The authors declare that they have no conflict of interest.\\

\noindent
\textbf{Funding statement}\\
This research was supported by the GNCS-INdAM 2026 project
\emph{``Metodi polinomiali e kernel per l'approssimazione da dati discreti e integrali con software OS''}
(CUP E53C25002010001).
The work of F. Nudo was funded by the European Union -- NextGenerationEU under the Italian National Recovery and Resilience Plan (PNRR), Mission 4, Component 2, Investment 1.2
\lq\lq Finanziamento di progetti presentati da giovani ricercatori\rq\rq,
pursuant to MUR Decree No.~47/2025.\\

\noindent
\textbf{Author Contributions}\\
Allal Guessab, Federico Nudo and Stefano Serra-Capizzano contributed equally to the conception, development, and writing of this manuscript.
All authors have read and approved the final version of the paper. For this reason, the order of authorship is alphabetical.\\

\noindent
\textbf{Acknowledgement}\\
This research was carried out as part of RITA \textquotedblleft Research ITalian network on Approximation'' and as part of the UMI group \lq\lq Teoria dell'Approssimazione e Applicazioni\rq\rq.\\

\noindent
\textbf{Data Availability}\\
No data were used in this study.

\bibliographystyle{unsrt}
\bibliography{bibliography}

\begin{thebibliography}{10}

\bibitem{Kak:2001:POC}
A.~C. Kak and M.~Slaney.
\newblock {\em Principles of computerized tomographic imaging}.
\newblock SIAM, 2001.

\bibitem{Natterer:2001:TMO}
F.~Natterer.
\newblock {\em The mathematics of computerized tomography}.
\newblock SIAM, 2001.

\bibitem{Palamodov:2016:RFI}
V.~P. Palamodov.
\newblock {\em Reconstruction from integral data}.
\newblock CRC Press Boca Raton, FL, 2016.

\bibitem{Bosner:2020:AOC}
T.~Bosner, B.~Crnkovi{\'c}, and J.~{\v{S}}kifi{\'c}.
\newblock Application of {CCC}--{S}choenberg operators on image resampling.
\newblock {\em BIT Numer. Math.}, 60:129--155, 2020.

\bibitem{Bruni:2025:PHO}
L.~Bruni~Bruno, F.~Dell'Accio, W.~Erb, and F.~Nudo.
\newblock Polynomial {H}istopolation {O}n {M}ock-{C}hebyshev {S}egments.
\newblock {\em J. Sci. Comput.}, 104:65, 2025.

\bibitem{Bruno:2026:BPH}
L.~Bruni~Bruno, F.~Dell'Accio, W.~Erb, and F.~Nudo.
\newblock Bivariate polynomial histopolation techniques on {P}adua, {F}ekete and {L}eja triangles.
\newblock {\em Adv. Comput. Math.}, 52(37), 2026.

\bibitem{DellAccio:2026:AGP}
F.~Dell’Accio, A.~Guessab, M.~Kbiri~Alaoui, and F.~Nudo.
\newblock A general probability density framework for local histopolation and weighted function reconstruction from mesh line integrals.
\newblock {\em Numer. Algorithms}, 2026.

\bibitem{Nudo:2026:FRU}
F.~Nudo.
\newblock Function reconstruction using a {J}acobi-weighted quadratic enriched histopolation method.
\newblock {\em Math. Comput. Simul.}, 245:512--529, 2026.

\bibitem{DellAccio:2026:NAM}
F.~Dell’Accio, A.~Guessab, G.~V. Milovanovi{\'c}, and F.~Nudo.
\newblock Nonconforming approximation methods for function reconstruction on general polygonal meshes via orthogonal polynomials.
\newblock {\em IMA J. Numer. Anal.}, 2026.

\bibitem{Demichelis:1995:GAO}
V.~Demichelis.
\newblock Graphic applications of some interpolating weighted mean functions.
\newblock {\em Rocky Mt. J. Math.}, pages 1277--1286, 1995.

\bibitem{Bose:1965:FSA}
A.~K. Bose.
\newblock Functions satisfying a weighted average property.
\newblock {\em Trans. Amer. Math. Soc.}, 118:472--487, 1965.

\bibitem{Diagana:2011:TEO}
T.~Diagana.
\newblock The existence of a weighted mean for almost periodic functions.
\newblock {\em Nonlinear Anal. Theory Methods Appl.}, 74:4269--4273, 2011.

\bibitem{Pesenson:2019:ASA}
I.~Z. Pesenson.
\newblock Average sampling and average splines on combinatorial graphs.
\newblock In {\em 2019 13th International conference on Sampling Theory and Applications (SampTA)}, pages 1--4. IEEE, 2019.

\bibitem{Pesenson:2019:WSA}
I.~Z. Pesenson.
\newblock Weighted sampling and weighted interpolation on combinatorial graphs.
\newblock {\em arXiv preprint arXiv:1905.02603}, 2019.

\bibitem{Guessab:2025:QWH}
A.~Guessab and F.~Nudo.
\newblock Quadratic {W}eighted {H}istopolation on {T}etrahedral {M}eshes with {P}robabilistic {D}egrees of {F}reedom.
\newblock {\em BIT Numer. Math. (accepted)}, 2026.

\bibitem{Garoni:2018:GLT1}
C.~Garoni and S.~Serra-Capizzano.
\newblock {\em Generalized {L}ocally {T}oeplitz {S}equences: {T}heory and {A}pplications: Volume {I}}.
\newblock Springer, Cham, 2018.

\bibitem{Garoni:2018:GLT2}
C.~Garoni and S.~Serra-Capizzano.
\newblock {\em Generalized {L}ocally {T}oeplitz {S}equences: {T}heory and {A}pplications: Volume {II}}.
\newblock Springer, Cham, 2018.

\bibitem{Dorostkar:2016:SAO}
A.~Dorostkar, M.~Neytcheva, and S.~Serra-Capizzano.
\newblock Spectral analysis of coupled {PDE}s and of their {S}chur complements via generalized locally {T}oeplitz sequences in 2{D}.
\newblock {\em Comput. Methods Appl. Mech. Eng.}, 309:74--105, 2016.

\bibitem{Salinelli:2016:EES}
E.~Salinelli, S.~Serra-Capizzano, and D.~Sesana.
\newblock Eigenvalue-eigenvector structure of {S}choenmakers--{C}offey matrices via {T}oeplitz technology and applications.
\newblock {\em Linear Algebra Appl.}, 491:138--160, 2016.

\bibitem{Garoni:2019:SBA}
C.~Garoni, H.~Speleers, S.~Ekstr{\"o}m, A.~Reali, S.~Serra-Capizzano, and T.~J.R. Hughes.
\newblock Symbol-{B}ased {A}nalysis of {F}inite {E}lement and {I}sogeometric {B}-{S}pline {D}iscretizations of {E}igenvalue {P}roblems: {E}xposition and {R}eview.
\newblock {\em Arch. Comput. Methods Eng.}, 26(5):1639--1690, 2019.

\bibitem{Benedusi:2024:MEC}
P.~Benedusi, P.~Ferrari, M.~E. Rognes, and S.~Serra-Capizzano.
\newblock Modeling excitable cells with the {EMI} equations: spectral analysis and iterative solution strategy.
\newblock {\em J. Sci. Comput.}, 98(3):58, 2024.

\bibitem{Szego:1975:OP}
G.~Szegő.
\newblock {\em Orthogonal polynomials}.
\newblock AMS Colloquium Publications, Vol. 23, 4th ed., 1975.

\bibitem{Milovanovic:1997:OPS}
G.~V. Milovanovi{\'c}.
\newblock Orthogonal polynomial systems and some applications.
\newblock {\em Inner Product Spaces and Applications, Pitman Res. Notes Math. Ser}, 376:115--182, 1997.

\bibitem{Erdelyi:1953:HTF}
A.~Erd{\'e}lyi, W.~Magnus, F.~Oberhettinger, and F.~G. Tricomi.
\newblock {\em Higher Transcendental Functions}.
\newblock Krieger, Malabar, Fla. Vol. 1, 1981.
\newblock Reprint of the 1953 edition.

\bibitem{Nevai:1994:GJW}
T.~Erd{\'e}lyi, A.~P. Magnus, and P.~Nevai.
\newblock Generalized {J}acobi weights, {C}hristoffel functions, and {J}acobi polynomials.
\newblock {\em SIAM J. Math. Anal.}, 25:602--614, 1994.

\bibitem{Bruno:2024:PIO}
L.~Bruni~Bruno and W.~Erb.
\newblock Polynomial interpolation of function averages on interval segments.
\newblock {\em SIAM J. Numer. Anal.}, 62:1759--1781, 2024.

\bibitem{Bruno:2025:OTC}
L.~Bruni~Bruno and S.~Serra-Capizzano.
\newblock On the conditioning of histopolation.
\newblock {\em arXiv preprint arXiv:2511.15395}, 2025.

\bibitem{Bhatia:1997:MAN}
R.~Bhatia.
\newblock {\em Matrix Analysis}.
\newblock Springer, New York, 1997.

\bibitem{Barbarino:2022:CAT}
G.~Barbarino, D.~Bianchi, and C.~Garoni.
\newblock Constructive approach to the monotone rearrangement of functions.
\newblock {\em Expo. Math.}, 40:155--175, 2022.

\bibitem{Ekstrom:2018:EAE}
S-E Ekstr{\"o}m and S.~Serra-Capizzano.
\newblock Eigenvalues and eigenvectors of banded {T}oeplitz matrices and the related symbols.
\newblock {\em Numer. Linear Algebra Appl.}, 25:e2137, 2018.

\bibitem{Mazza:2019:SAA}
M.~Mazza, A.~Ratnani, and S.~Serra-Capizzano.
\newblock Spectral analysis and spectral symbol for the 2{D} curl-curl (stabilized) operator with applications to the related iterative solutions.
\newblock {\em Math. Comput.}, 88:1155--1188, 2019.

\bibitem{Serra:2001:DRO}
S.~Serra-Capizzano.
\newblock Distribution results on the algebra generated by {T}oeplitz sequences: a finite-dimensional approach.
\newblock {\em Linear Algebra Appl.}, 328:121--130, 2001.

\bibitem{Serra:2003:GLT}
S.~Serra-Capizzano.
\newblock Generalized locally {T}oeplitz sequences: spectral analysis and applications to discretized partial differential equations.
\newblock {\em Special issue on structured matrices: analysis, algorithms and applications (Cortona, 2000), Linear Algebra Appl.}, 366:371--402, 2003.

\bibitem{Serra:2006:TGC}
S.~Serra-Capizzano.
\newblock The {GLT} class as a generalized {F}ourier analysis and applications.
\newblock {\em Linear Algebra Appl.}, 419:180--233, 2006.

\bibitem{Tilli:1998:LTS}
P.~Tilli.
\newblock Locally {T}oeplitz sequences: spectral properties and applications.
\newblock {\em Linear Algebra Appl.}, 278:91--120, 1998.

\bibitem{Barbarino:2020:BGL2}
G.~Barbarino, C.~Garoni, and S.~Serra-Capizzano.
\newblock Block generalized locally {T}oeplitz sequences: theory and applications in the unidimensional case.
\newblock {\em Electron. Trans. Numer. Anal.}, 53:28--112, 2020.

\bibitem{Barbarino:2020:BGL1}
G.~Barbarino, C.~Garoni, and S.~Serra-Capizzano.
\newblock Block generalized locally {T}oeplitz sequences: theory and applications in the multidimensional case.
\newblock {\em Electron. Trans. Numer. Anal.}, 53:113--216, 2020.

\bibitem{Serra:2003:AOP}
S.~Serra-Capizzano and C.~Tablino-Possio.
\newblock Analysis of preconditioning strategies for collocation linear systems.
\newblock {\em Linear Algebra Appl.}, 369:41--75, 2003.

\bibitem{Benedusi:2026:AOE}
P.~Benedusi, S.~Riva, L.~Belluzzi, and S.~Serra-Capizzano.
\newblock Analysis of eigenvalue clustering leads to optimal scaling in numerical radiative transfer.
\newblock {\em arXiv preprint arXiv:2602.21958}, 2026.

\bibitem{Carl:2009:OWI}
B.~Carl.
\newblock On a {W}eyl inequality of operators in {B}anach spaces.
\newblock {\em Proc. Am. Math. Soc.}, 137:155--159, 2009.

\end{thebibliography}

\end{document}